\documentclass[draft, 12pt, reqno]{amsart} 
\usepackage{amssymb,latexsym,amsfonts,verbatim,amscd,ifthen} 
\usepackage{color}
\usepackage{relsize}

\usepackage[utf8]{inputenc}
\usepackage{mathrsfs}
\newcommand{\C}{{\mathbb C}}

\newcommand{\Po}{{\mathcal P}}

\numberwithin{equation}{section}

\theoremstyle{plain}

\newtheorem{theorem}[equation]{Theorem}
\newtheorem{lemma}[equation]{Lemma}
\newtheorem{proposition}[equation]{Proposition}

\newtheorem*{definition*}{Definition}

\theoremstyle{definition}

\begin{document}

\title[Spectral analysis and representations]{Spectral analysis near regular point of reducibility and representations of Coxeter groups }
\author{Michael I Stessin}
\address{Department of Mathematics and Statistics \\
University at Albany \\
Albany, NY 12222}
\email{mstessin@albany.edu}

\begin{abstract}
For a tuple	of square matrices $A_1,...,A_n$ the determinantal hypersurface is defined as
\begin{eqnarray*}
&\sigma(A_1,...,A_n)=  \\
&\Big\{[x_1:\cdots :x_n]\in \C{\mathbb P}^{n-1}:  det(x_1A_1+\cdots +x_nA_n)=0\Big \}.
\end{eqnarray*}
In this paper we develop a local spectral analysis near a regular point of reducibility of a determinantal hypersurface. We prove a rigidity type theorem for representations of Coxeter groups  as an application .\\
\end{abstract}
\maketitle


\noindent Data avalability statement: Data sharing not applicable.

\vspace{.5cm}

\noindent \Large \textit{Dedicated to the memory of J$\ddot{o}$rg Eschmeier} \normalsize

\section{\textbf{Introduction and statements of results}} 

\vspace{.2cm}

Given $n$ square $N\times N$ matrices, the determinant of their linear combination, 
$det(x_1A_1+\cdots +x_nA_n)$, (such linear combination is called \textit{a pencil}) is a homogeneous polynomial of degree $N$ in variables $x_1,...,x_n$. Zeros of this polynomial form an algebraic manifold in the projective space $\C{\mathbb P}^{n-1}$. This manifold is called the \textbf{determinantal manifold} (or \textbf{determinantal hypersurface}) for the tuple $(A_1,\dots,A_n)$. We use the following notation 
$$\sigma(A_1,...,A_n)=  
\Big\{[x_1:\cdots :x_n]\in \C{\mathbb P}^{n-1}: \ det(x_1A_1+\cdots +x_nA_n)=0\Big\}.$$

An infinite dimensional analog of determinantal manifold, called \textbf{projective joint spectrum} of a tuple of operators acting on a Hilbert space was introduced in \cite{Y}. For operators $A_1,\dots,A_n$ acting on a Hilber space $H$ it is defined by 
\begin{align*}
&\sigma(A_1,\dots, A_n) \\
&=\Big\{ [x_1:\cdots :x_n]\in \C{\mathbb P}^{n-1}: \ x_1A_1+\cdots +x_nA_n \ \mbox{is not invertible}\Big \}.
\end{align*}

To avoid trivial redundancies it is frequently assumed that at least one of the operators is invertible, and, thus, can be assumed to be the identity. In what follows we always assume that $A_{n+1}=I$. It was shown in \cite{ST} that a lot of information can be obtained from the part of the joint spectrum that lies in the chart $\{ x_{n+1}\neq 0\}$ (and, therefore, we can take $x_{n+1}=-1$). We call this part the \textbf{proper joint spectrum} and denote it $\sigma_p(A_1,\dots,A_n)$,
\begin{eqnarray*}
&\sigma_p(A_1,\dots,A_n)=\Big\{ (x_1,...,x_n)\in \C^n: \ x_1A_1+\cdots +x_nA_n-I\\
&\mbox{is not invertible}\Big \}.
\end{eqnarray*}

When the dimension of $H$ is finite the projective joint spectrum coincides with the corresponding determinantal manifold. This is why we use the same notation in both cases. In this paper we concentrate on the finite dimensional case, so we will use ``determinantal manifold" and ``projective joint spectrum" interchangiably. Of course, in the finite dimensional case
\small$$ \sigma_p(A_1,\dots,A_n)=\Big \{ (x_1,,\dots,,x_n)\in \C^n:  det(x_1A_1+\cdots+x_nA_n-I)=0\Big \}.$$\normalsize

Matrix pencils have been under investigation for a long time. Ultimately this line of research led Frobenius to laying out the foundation of representation theory, cf \cite{Fr1}-\cite{Fr3}. 

One of the basic questions regarding determinantal manifolds was when a hypersurface in the projective space admits a determinantal representation. The number of publications in this area is very substantial. Without trying to give an exhaustive account of the results here we just refer a reader to \cite{CAT}, \cite{D1} - \cite{Do}, \cite{HMV}, \cite{HV}, \cite{KV}, \cite{KVo}, \cite{V} and references there. 

In this paper we investigate joint spectra from a different angle: given that an algebraic hypersurface in $\C{\mathbb P}^n$ has a determinantal representation, what does the geometry of the surface tell us about relations between operators in a representing pencil? This line of research was paid much less attention until recently. The only early result we can mention is the one by Motzkin and Taussky \cite{MT}. The situation changed in the last decade when joint spectra have been intensely scrutinized  exactly from this point of view (see \cite{BCY}, \cite{CY},  \cite{CSZ}, \cite{CST}, \cite{DY1}, \cite{DY2}, \cite{GY} - \cite{GYa}, \cite{MQW},  \cite{PS} - \cite{SYZ}, \cite{Y}, and references there).

In particular, papers \cite{PS} and \cite{ST} contain a local spectral analysis near the reciprocal of a spectral point of one of the operators when this point does not belong to the singular locus of the joint spectrum. Results of this analysis led to a spectral characterization of representations of non-special finite Coxeter groups in \cite{CST} and of Hadamard matrices of Fourier type in \cite{PS}. Non-singularity was essential in these cases. Here we concentrate on the local spectral analysis in a neighborhood of a point at which the joint spectrum is reducible, and, therefore, this point belongs to the singular locus of the joint spectrum. One of the difficulties of carrying out spectral analysis in this setting is that, in general, spectral projections related to operator families analytically depending on parameters (in our setting parameters are coordinates $(x_1,...,x_n)$ of a point in the joint spectrum) might blow up  when approaching a singular point. It turned out that if one of the matrices is diagonal, and the joint spectrum satisfies regularity conditions  a) and b) introduced below, then the limit projections exist when approaching a singular point along a curve in a spectral component that is non-tangential to the singular locus. 

Write
\small \begin{equation}\label{projective_spectrum}
\sigma(A_1,...,A_n,-I)=\big \{ R_1(x_1,..,x_n,x_{n+1})^{m_1}...R_s(x_1,...,x_n,x_{n+1})^{m_s},=0\big\}, 
\end{equation}\normalsize
where $R_1,..R_s$ are irreducible polynomials. 

Suppose that $\lambda\in \sigma(A_1)$.  We introduce the regularity conditions at $\lambda$, which we mentioned above, separately for $\lambda \neq 0$ and $\lambda=0$.. 
 
\noindent 1). $\lambda \neq 0$. 

In this case 
\begin{equation}\label{proper_spectrum}
\sigma_p(A_1,...,A_n)= \big\{ R_1(x_1,..,x_n,1)^{m_1}...R_s(x_1,...,x_n,1)^{m_s}=0\big\}. 	
\end{equation}

Let us denoted by $I_\lambda$ the set of indexes $1\leq j\leq s$ corresponding to the components of the proper projective spectrum passing through $(1/\lambda,0,...,0,-1)$:
\begin{equation}\label{K_lambda}
I_\lambda=\big \{	1\leq j\leq s: \ R_j(1/\lambda,0...,0,1)=0 \big \}.
\end{equation}

We call the following conditions \textbf{ the regularity conditions at $\lambda$:}
  \begin{itemize} 
  \item[a)] $\frac{\partial R_j}{\partial x_1}(1/\lambda,0,...,0,1)\neq 0, \ j\in I_\lambda$
\item[b)] For every pair $i\neq j, \ i,j \in  I_\lambda$ the vectors  $\partial R_i(1/\lambda,0,...,0,1)$ and $\partial R_j(1/\lambda,0,...,0,1)$ are not proportional, where 
$$\partial R_l=\bigg(\frac{\partial R_l}{\partial x_1},\dots, \frac{\partial R_l}{\partial x_n} \bigg).$$
 
\end{itemize} 

\noindent 2). $\lambda=0$.

In this case we pass to the chart $\{x_1\neq 0\}$ (and we take $x_1=1$ here). Let us denote by $\tilde{\sigma}_{p,1} (A_1,...,A_n)$ the part of the projective joint spectrum that lies in this chart:
$$\tilde{\sigma}_{p,1}(A_1,...,A_n)=\big \{(x_2,...,x_n,x_{n+1})\in \C^n: \ A_1+\bigg(\sum_{j=2}^n x_jA_j\bigg ) -x_{n+1}I  $$
$$\mbox{ is not invertible} \bigg \}. $$
If $0\in \sigma(A_1)$, then $0\in \tilde{\sigma}_{p,1}(A_1,...,A_n)$,  and  $\tilde{\sigma}_{p,1}(A_1,...,A_n)$ is given by
 \begin{equation}\label{spectrum_near_0}
\tilde{\sigma}_{p,1}(A_1,...,A_n)= \{ R_1(1,x_2,..,x_{n+1})^{m_1}...R_s(1,x_2,...,x_{n+1})^{m_s}=0\}. 
\end{equation}
 Similarly, we denote by $I_0$ the set of indexes corresponding to the components of \eqref{spectrum_near_0} passing through the origin
 \begin{equation}\label{spectrum at 0}
 I_0=\big\{ 1\leq j\leq s : \ R_j(1,0,...,0)=0 \big \}.	
 \end{equation}

 The regularity conditions here are:
\begin{itemize} 
\item[$\tilde{a}$)] $\frac{\partial R_j}{\partial x_{n+1}}(0)\neq 0, \ j\in I_0.$
\item[$\tilde{b}$)] For every pair $i\neq j, \ i,j\in I_0$ the vectors $\partial R_i$ and $\partial R_j$ at the origin are not proportional, and here $\partial R_l=\big(\frac{\partial R_l}{\partial x_2},\dots,\frac{\partial R_l}{\partial x_{n+1}}\big) $.
\end{itemize}
 In both cases $\lambda\neq 0$ and $\lambda=0$  we write $\hat{x}=(x_2,...,x_n)$.
\vspace{.2cm}

Let ${\mathcal O}$ be a small neighborhood of $(1/\lambda,0,...,0)\in \C^n$, if $\lambda\neq 0$, or of the origin in $\C^n$, if $\lambda=0$. Write 
\begin{eqnarray*}
\tilde{R}_j(x_1,...,x_n)=R_j(x_1,...,x_n,1) \ \mbox{if} \ \lambda\neq 0 \\
\tilde{R}_j(x_2,...,x_n,x_{n+1})=R(1,x_2,...,x_{n+1}) \ \mbox{if} \ \lambda=0.
\end{eqnarray*} 
 The singular locus of the surface $\big\{\prod_{j\in I_\lambda} \tilde{R}_j=0\big\} \cap {\mathcal O}$ is an algebraic manifold in ${\mathcal O}$ which we denote by ${\mathcal U}$. Of course, geometrically ${\mathcal U}$ coincides with singular locus of $\sigma_p(A_1,...,A_n)$ or $\tilde{\sigma}_{p,1}(A_1,...,A_n)$   in ${\mathcal O}$, but they have different multiplicities: the multiplicity of each regular point of $\big\{\prod_{j\in I_\lambda}R_j=0\big\}$ is 1. 

Conditions a) and b) $\big(\tilde{a}) $ \ and $\tilde{b})$  respectively$\big) $ imply that ${\mathcal U}$ has codimension greater than 1. Let us denote by $\tilde{{\mathcal U}}$ the orthogonal projection of ${\mathcal U}$ onto the hyperplain $\{x_1=0\}$  if $\lambda\neq 0$, and on the hyperplain $\{x_{n+1}=0\}$, if $\lambda=0$,     
$$\tilde{{\mathcal U}}=\{ (x_2,...,x_n): \ \exists (x_1,x_2,...,x_n)\in {\mathcal U} \}, \ \lambda\neq 0,$$ 
and  
$$\tilde{{\mathcal U}}=\{ (x_2,...,x_n): \ \exists (x_2,...,x_{n+1})\in {\mathcal U} \}, \ \lambda=0,$$ 
then  $\tilde{{\mathcal U}}$ is  a set of positive codimension in the plain $\{x_1=0\}$   for $\lambda\neq 0$, and in the plain $\{x_{n+1}=0\}$ when $\lambda=0$.

Let $ j\in I_\lambda$. Since by condition a) when $\lambda\neq 0$,  $\big($by $\tilde{a})$, when $\lambda=0\big)$, \  $\frac{\partial R_j}{\partial x_1}(1/\lambda,0,...,0)\neq 0$, \ $\big( \frac{\partial R_j}{\partial x_{n+1}}(0,...,0)\neq 0\big )$, by the implicit function theorem, in a small neighborhood of the origin in $\C^{n-1}$, \ $ {\mathcal O}(\epsilon_j)$, the first coordinate $x_1$ ( the $(n+1)$-st coordinate $x_{n+1}$, respectively) can be expressed as an analytic function $x_{1,\lambda,j}(\hat{x})$ \ ($x_{n+1,0,j}(\hat{x})$ when $\lambda=0$), of $\hat{x}=(x_2,...,x_{n})$ satisfying 
\begin{align}
&x_{1,\lambda,j}(0)=1/\lambda, \  \tilde{R}_j(x_{1,\lambda,j}(\hat{x}),\hat{x})=0, \ \hat{x}\in {\mathcal O}(\epsilon_j), \ \mbox{if} \ \lambda\neq 0, \nonumber \\
& \label{x 1  lambda j} \\
&x_{n+1,0,j}(0)=0, \ \tilde{R}_j(\hat{x}, x_{n+1,0,j}(\hat{x}))=0, \ \hat{x}\in {\mathcal O}(\epsilon_j), \ \mbox{if} \ \lambda=0. \nonumber 
\end{align}
Conditions a) and b) \  ( $\tilde{a})$ and $\tilde{b})$ ) imply that in ${\mathcal O}(\epsilon_j)$ \ we have $1-x_{1,\lambda,j}=O(|\hat{x}|) $ (respectively $x_{n+1,0,j}=O(|\hat{x}|)$ for $\lambda=0$), so that there are $d_j>0$ such that
\begin{eqnarray}
&|1-x_{1,\lambda,j}(\hat{x})|\leq d_j |\hat{x}|, \ \hat{x}\in {\mathcal O}(\epsilon_j), \ \lambda\neq 0 \nonumber \\
& \label{estimate5} \\
& |x_{n+1,0,j}(\hat{x})|\leq d_j |\hat{x}|, \  \hat{x}\in {\mathcal O}(\epsilon_j), \     \lambda=0 \nonumber.
\end{eqnarray}

\vspace{.2cm}

Let $\lambda\in \sigma(A_1),  \ \lambda \neq 0,  \ j\in I_\lambda$.  Suppose that 
\begin{equation}\label{point on j component}
y_{j,\lambda}(\hat{x})=(x_{1,\lambda,j}(\hat{x}),\hat{x})   
  	\end{equation}
 
is a regular spectral point 
which belongs to the $j$-th component $\{ \tilde{R}_j=0\}$, and $\delta_j(x)$ is so small  that no eigenvalue of $x_{1,\lambda,j}(\hat{x})A_1+x_2A_2+\cdots +x_nA_n$ other than 1 is in the $\delta_j(x)$-neighborhood of 1, we denote by $P_{j,\lambda}(\hat{x})$ the projection
\begin{equation}\label{projection}
P_{j,\lambda}(\hat{x})=\frac{1}{2\pi i} \int_{|w-1|=\delta_j(x)} \bigg(w-(x_{1,\lambda,j}(\hat{x})A_1+...+x_nA_n)\bigg)^{-1}dw .
\end{equation}
If $\lambda=0$ is in the spectrum of $A_1$, \ $x=(\hat{x},x_{n+1,0,j}(\hat{x})$ is in the $j$-th component of $\tilde{\sigma}_{p,1}(A_1,...,A_n)$, and $\delta_j(x)$ is small enough so that no non-zero eigenvalue of $A_1+x_2A_2+,\dots,+x_nA_n-x_{n+1,,0,j}(\hat{x})I)$ is in the $\delta_j(x)$-neighborhood of the origin, the corresponding projection $P_{j,0}(\hat{x})$ is given by
\begin{equation}\label{projection for 0}
P_{j,0}(\hat{x})=	\frac{1}{2\pi i} \int_{|w|=\delta_j(x)} \bigg( (w+x_{n+1,0,j}(\hat{x}))I-A_1-...-x_nA_n)\bigg)^{-1}dw. 
\end{equation}

Our first result is

	\begin{theorem}\label{limit_projections}
Let $(A_1,...,A_n)$ be a matrix tuple, such that $A_1$ is a normal matrix. Suppose that $\lambda\in \sigma(A_1)$, and  that $\sigma_p(A_1,\dots,A_n)$, if $\lambda\neq 0$, or  $\tilde{\sigma}_{p,1}(A_1,\dots,A_n)$, if $\lambda=0$, satisfies conditions a) and b) if $\lambda \neq 0$ and respctively conditions $\tilde{a})$  and \ $\tilde{b})$ if $\lambda=0$ . Fix $\hat{x}$ such that the line $\{t\hat{x}: t\in \C\}$ is not tangent to  any component of  \ $\tilde{{\mathcal U}}$ at the origin. Then for $j\in I_\lambda$ each projection $P_{j, \lambda}(t\hat{x})$ for $\lambda\neq 0$ and $P_{j,0}(t\hat{x})$ for $\lambda=0$,  can be extended to $t=0$, so,  it becomes an analytic function of $t$ in some neighborhood ${\mathcal O}$ of the origin.  
\end{theorem}

We denote this limit projection by $\Po_{j,\lambda}(\hat{x})$,
\begin{eqnarray}
&\Po_{j,\lambda}(\hat{x})=\displaystyle\lim_{t\to 0} P_{j,\lambda}(t\hat{x}), \ \lambda\neq 0, \nonumber \\
& 	\label{limit1} \\
&\Po_{j,0}(\hat{x})=\displaystyle\lim_{t\to 0} P_{j,0}( t\hat{x}), \ \lambda = 0. \nonumber 
\end{eqnarray}

Section \ref{component projections} is devoted to the proof of this result.

\vspace{.2cm}

We express certain pairwise relations between operators in the tuple in terms of these limit projections in section \ref{relations1}.
The main results here are the following Theorems formulated for the pair $(A_1,A_2)$. Of course, when we have only 2 matrices in the tuple, that is when $n=2$, the dependance on $\hat{x}$ is redundant, and we simply write $\Po_{j,\lambda}$.

Let $A_1$ be a normal matrix and
\begin{equation}\label{A_1 decomposition}
A_1=\sum_{\lambda \in \sigma(A_1)} \lambda {\mathscr P}_\lambda	
\end{equation}
be its spectral resolution. For $\lambda \in \sigma(A_1)$  write
\begin{equation}\label{T}
T_\lambda=\sum_{\mu \in \sigma(A_1),\mu\neq \lambda}	\frac{{\mathscr P}_\mu}{\lambda-\mu}.	
\end{equation}

\begin{theorem}\label{first two moments}
Let $A_1$ and $A_2$ be $N\times N$ matrices with $A_1$ being normal. Suppose that  the regularity conditions are satisfied at every spectral point of $A_1$. Further, suppose that $\lambda\in \sigma(A_1) $, and that the multiplicities $m_j, \ j\in I_\lambda$ are equal to 1. Then     the following relations hold:
\begin{eqnarray}
& \Po_{j,\lambda}A_2\Po_{j,\lambda}+x_{1,\lambda,j}^\prime(0)\Po_{j,\lambda}=0,, \ \lambda \neq 0, \label{first moment}\\
& \Po_{j,0 }A_2\Po_{j,0} - x_{3,0,j}^\prime(0)\Po_{j,0}=0, \ \lambda=0 \label{first moment at 0} \\
&\Po_{j,\lambda}A_2T_\lambda A_2\Po_{j,\lambda}+\frac{x_{1,\lambda,j}^{\prime \prime}(0)}{2}\Po_{j,\lambda}=0, \ \lambda \neq 0,	\label{moment2} \\
&\Po_{j,0}A_2T_0A_2\Po_{j,0}-\frac{x_{3,0,j}^{\prime \prime}(0)}{2}\Po_{j,0}=0, \ \lambda=0. \label{moment2 at 0}
\end{eqnarray}

\end{theorem}	

\vspace{.2cm}

\vspace{.2cm}

\vspace{.3cm}

\begin{theorem}\label{A^2}
Let $\lambda \in \sigma(A_1), \ \lambda\neq 0$, and both pairs of matrices $(A_1,A_2)$ and $(A_1,A_1A_2)$ satisfy conditions of Theorem \ref{first two moments}. Then
\begin{equation}\label{for the square}
\Po_{j,\lambda}A_2^2\Po_{j,\lambda}=\frac{ z_{1,\lambda,j}^{\prime \prime}(0)+2\lambda\Big(x_{1,\lambda,j}^\prime(0)\Big)^2-\lambda ^2x_{1,\lambda,j}^{\prime \prime}(0)}{2\lambda}\Po_{j,\lambda},	
\end{equation}
where $z_1=z_{1,\lambda,j}(z_2)$ is the local representation of the $j$-th component of $\sigma_p(A_1,A_1A_2)$, \ $\{\tilde{R}_j(z_1,z_2)=0\}$, near $(1/\lambda,0))$.	
\end{theorem}

\vspace{.2cm}

Finally, in section \ref{Coxeter} we give an application of our technique to representations of Coxeter groups. Here we prove the following result which might be viewed as a rigidity type theorem for representations of Coxeter groups.

\begin{theorem}\label{for Coxeter}
Let $G$ be a Coxeter group with Coxeter generators $g_1,...,g_n$, and let $\rho$ be a $m$-dimensional linear representation of $G$. For every $2\leq i\leq n $ denote by $\tilde{\rho}_i$ the representation of the Dihedral group  $D_{1i}$ generated by $g_1$ and $g_i$, which is induced by $\rho$. Assume that  the following condition is satisfied:
 \begin{itemize}
\item[$(\ast)$] For $2\leq i\leq n$	no irreducibble representation of $D_{1i}$ is included in the decomposition of $\tilde{\rho}_i$ with coefficient bigger than 1.
\end{itemize}
Suppose that $A_1,...,A_n$ are complex matrices in $M(N)$ such that $A_1$ is normal, $\parallel A_j\parallel = 1, \ j=2,...,n$, and conditions a) - b), \ $\tilde{a})$ - $\tilde{b})$ are satisfied for the pairs $(A_1,A_j)$ and $(A_1,A_1A_j), \ j=2,...,n$ at every spectral point of $A_1$. Also suppose that
\begin{itemize}
\item[(I)] $\sigma_p(A_1,...,A_n)\supset \sigma_p	(\rho(g_1),...,\rho(g_n))$.
\item[(II)] $\exists \epsilon >0$ such that for every point $\zeta_j^+=(0,...,\underbrace{1}_J,0,...,0)$ and   $\zeta_j^-=(0,...,\underbrace{-1}_J,0,...,0), \ j=1,...,n$ we have
\begin{align*}
&\sigma_p\big(A_1,...,A_n, A_1A_2,...,A_1A_n\big)\cap {\mathcal O}_\epsilon (\zeta_j^\pm) \\
&=\sigma_p\big(\rho(g_1),...,\rho(g_n), \rho(g_1)\rho(g_2),...,\rho(g_1)\rho(g_n)\big)\cap  {\mathcal O}_\epsilon (\zeta_j^\pm).
\end{align*} 
\end{itemize}
Then
\begin{itemize}
\item[1)] There exists an $m$-dimensional subspace $L$  of \ $\C^N$ invariant under the action of each matrix $A_j, \ j=1,...,n$.
\item[2)] Restrictions of $A_j, \ j=1,...,n$ to $L$ are unitary, self-adjoint, and generate a representation, $\hat{\rho}$ of  \ $G$, and  
 \begin{equation}\label{same spectra}
 \sigma_p \Big(A_1\Big |_L,...,A_n\Big |_L\Big)=\sigma_p\Big (\rho(g_1),...,\rho(g_n)\Big )
 \end{equation}.
\item[3)] If $G$ is finite, non-special Coxeter group (that is either a Dihedral, or of types A,B, or D), then $\hat{\rho}$ is unitary equivalent to $\rho$.
\end{itemize}

\end{theorem}

\vspace{.2cm}

\vspace{.4cm}

\section{\textbf{Limit projections along components.  \\ Proof of Theorem \ref{limit_projections} }}\label{component projections}.

Observe that under conditions a), b) and $\tilde{a})$, \ $\tilde{b})$ we have:
\begin{itemize} 
\item $M_\lambda =\sum_{j\in I_\lambda} m_j$, is the multiplicity of $\lambda$ in $\sigma(A_1)$. Here $m_j$ are exponents from \eqref{projective_spectrum}.
\item If $\lambda\neq 0$ and ${\mathcal O}$ is a small enough neighborhood of $(1/\lambda,0\dots,0)$, then for every $x=(x_1,...,x_n)$ sufficiently close to $(1/\lambda,0,...,0)$ and every $j\in I_\lambda$ the line 
through the origin and $x$ has only one point of intersection with the hypersurface $\{\tilde{R}_j=0 \}$ which lies in ${\mathcal O}$. Similarly, if $\lambda=0$ and ${\mathcal O}$ is a small enough neighborhood of the origin,  every line parallel to the $x_{n+1}$-axis and passing through a point $\hat{x}$ close to 0  has only one point of intersection with the hypersurface $\{\tilde{R}_j=0 \}, \ j \in I_0$ which lies in ${\mathcal O}$.
\end{itemize}
``The component $\{\tilde{R}_j=0\}$ of $\sigma_p(A_1,...,A_n)$ ($\tilde{\sigma}_{p,1}(A_1,...,A_n)$) has multiplicity $m_j$" means that for every regular point in this component that does not belong to any other component, the rank of the projection \eqref{projection}
is equal to $m_j$. As mentioned in the introduction, $\delta_j$ in  \eqref{projection} is so small  that $\delta_j$-neighborhood of 1 does not contain any eigenvalues of $x_{1j}(\hat{x})A_1+...+x_nA_n$ different from 1 (respectively $\delta_j$-neighborhood of 0 does not contain non-trivial eigenvalues of $A_1+x_2A_2+...-x_{n+1}I$). The existence of such $\delta_j$ follows from conditions b) and $\tilde{b}$).

In general, even under the regularity conditions projection $P_{j,\lambda}(\hat{x})$ given by \eqref{projection} might "blow up" as the point $\hat{x}$ approaches 0 . A simple example is:
$$A_1=\left [ \begin{array}{cc} 1 & 1 \\ 0 & 1 \end{array} \right ], \ A_2=\left [ \begin{array}{cc} 1 &0 \\ 0 & -1 \end{array} \right ]. $$
In this case $\hat{x}=x_2$, and  the joint spectrum, $\sigma_p(A_1,A_2)$ is
$$\{x_1+x_2-1=0\} \cup \{ x_1-x_2-1=0\},$$
so conditions a) and b) are satisfied at (1,0). When a point $(x_1,x_2)$ approaches $(1,0)$ along the components $\{x_1+x_2-1=0\} $ and $\{ x_1-x_2-1=0\}$, the corresponding projections
$P_{1,1}(x_2)$ and $P_{2,1}(x_2)$ are given by
$$ P_{1,1}(x_2)=\left [ \begin{array}{cc} 1 & \frac{1-x_2}{2x_2} \\0 & 0  \end{array}\right ], \ P_{2,1}(x_2)= \left [ \begin{array}{cc} 0 & -\frac{1+x_2}{2x_2} \\0 & 1\end{array} \right ].$$ 
We see that the norm of each projection goes to $\infty$ as $x_2\to 0$.

Theorem \ref{limit_projections} claims that this does not happen in the case when $A_1$ is normal.

\vspace{.3cm}

\Large \textit{Proof of Theorem \ref{limit_projections}}. \normalsize

\vspace{.2cm}
 
1. $\lambda \neq 0$. 

\noindent Since scaling preserves conditions a) and b), we may replace $A_1$ with $A_1/\lambda$ and assume that $\lambda=1$. To slightly simplify the notation, in this proof we will write $x_{1j}(\hat{x}), \ j\in I_1$ instead of $x_{1,1,j}(\hat{x})$ and $P_j(\hat{x})$ instead of $P_{j,1}( \hat{x})$. 

Let $|\hat{x}|< \min \{ \epsilon_j: \ j=1,...,r\}$, where $\epsilon_j$ are constants from \eqref{estimate5}, and suppose that ${\mathcal O}=\cap_j {\mathcal O}(\epsilon_j)$.

Denote by $\Gamma_{ij}=\{ \hat{x}\in {\mathcal O}: \ x_{1i}(\hat{x})=x_{1j}(\hat{x}) \}, \ i\neq j$. Then ${\mathcal U}=\cup_{i,j=2}^k \Gamma_{ij}$, and $\tilde{{\mathcal U}}=\cup_{i,j=2}^k \tilde{\Gamma}_{ij}$, where $\tilde{\Gamma}_{ij}$ is the projection of $\Gamma_{ij}$ onto $\{ x_1=0\}$.

Condition b) implies that the tangent planes at the origin $\hat{x}=0$ to the surfaces $\{x_1=x_{1j}(\hat{x}), \ j\in I_1\}$ are pairwise different. These tangent planes are  given by
$$x_1-1=\sum_{l=2}^n \frac{\partial x_{1j}}{\partial x_l}(0)x_l, $$
so that the tangent plain in $\{ x_1=0\}$ to $\tilde{\Gamma}_{ij}$ is given by
\begin{equation}\label{tangent}
\left \{\hat{x}: \ \sum_{l=2}^n \left ( \frac{\partial x_{1j}}{\partial x_l}(0)-\frac{\partial x_{1j}}{\partial x_l}(0) \right )x_l=0 \right \} .
\end{equation}

If a point $\hat{x}$ does not belong to any of the plains given by \eqref{tangent},
then there exists  $a>0$ and  $\delta>0$ such that for every $z\in \C$ with $|z|<\delta$ we have 
\begin{equation}\label{distance}
|x_{1i}(z\hat{x})-x_{1j}(z\hat{x}|\geq a|z|, \ i\neq j. 
\end{equation}
The constant $a$ continuously depends on
$$\min_{i,j} \left \{ \left | \sum_{k=2}^n \left ( \frac{\partial x_{1j}}{\partial x_k}(0)-\frac{\partial x_{1j}}{\partial x_k}(0) \right )x_k \right | \right \}, $$
and, therefore, may be chosen such that \eqref{distance} holds uniformly in a neighborhood of $\hat{x}$. 

Recall that for each $j\in I_\lambda$ the point $y_{j,\lambda}(\hat{x})$ in the component $\{\tilde{R}_j=0\}$ was defined by  \eqref{point on j component}. In our case $\lambda=1$, and instead of $y_{j,1}(\hat{x})$ we will write $y_j(\hat{x})$.  For each $i\in I_1$  and $z\in \C$ close to 0 define $\tau_{ij}(z)\in \C $ by the condition that it is the closest to 1 root 
of the following equation  in $\tau$ 
\begin{equation}\label{root equation}
\tilde{R}_i(\tau y_j(z\hat{x}))=0. 
\end{equation}
As mentioned above, conditions a) and  b)  imply that every line in $\C^n$ passing through the origin and close to the $x_1$-axis has only one point of intersection with the hypersurface $\{ \tilde{R}_i=0\}$ that is close to $(1,0,...,0)$. When this line is the line passing through the origin and $y_j(z\hat{x})$, this point determines $\tau_{ij}(z)$ in the following way. Write the homogeneous decomposition of $\tilde{R}_j$:
$$\tilde{R}_i(x)=\tilde{R}_i(x_1,\hat{x})=\sum_{l=0}^{t_i} S_{il}(x), $$
where $S_{il}(x)$ is a homogeneous polynomial of degree $l$ in $x_1,...,x_n$ and $t_i$ is the degree of $\tilde{R}_i$. We always assume that $\tilde{R}_0=-1$. Then $\tau_{ij}(z)$ is the only root of the equation (in $\tau$)
$$
\sum_{l=0}^{t_i} \tau^l S_{il}(y_j(z\hat{x}))=0
$$
which is close to 1. Since this root has multiplicity 1, $\tau_{ij}(z)$ is an analytic functiion of $z$ in a small neighborhood of the origin in $\C$, which we denote by $U$.

Of course,  \ $\tau_{jj}(z)=1, \ \tau_{ij}(z)\neq 1$ for $i\neq j, \ z\neq 0$, and, of courses, $\tilde{R}_i(\tau_{ij}(z)y_j(z\hat{x}))=0$ for all $z\in U$. It is easy to see that \eqref{distance} implies that there exist  positive constants $c$ and $C$ such that for all $j\neq i$
\begin{equation}\label{spectral_distance}
c|z|\leq |\tau_{ij}(z)-1|\leq C|z|. 
\end{equation}
Using the notation $A(x)=x_1A_1+...+x_nA_n$, it follows from the definition of the joint spectrum that the reciprocals $\mu_{ij}(z)=\frac{1}{\tau_{ij}(z)}$ are eigenvalues of $A(y_j(z\hat{x}))$. Now, \eqref{spectral_distance} implies that a similar estimate (with different constants) holds for the eigenvalues $\mu_{ij}$:
\begin{equation}\label{spectral_distance1}
c_1|z|\leq |\mu_{ij}(z)-1|\leq C_1|z|, \ j\neq i.	
\end{equation}
This shows that $\delta_j$ in \eqref{projection} should be chosen less than  $c_1|z|$ (of course, the integral \eqref{projection} is the same for all $\delta_j<c_1|z|$).
We call it $\delta_j(z)$.

Let $e_1,...,e_N$ be an eigenbasis for $A_1$, where  $e_1,...,e_{M_1}$ are eigenvectors with eigenvalue 1, and $e_{M_1+1},...,e_N$ correspond to eigenvalues different from 1, so that in the basis $e_1,...,e_N$ the matrix $A_1$ is diagonal with the first $M_1$ diagonal entries equal to 1 (recall, $M_1=\sum_{j\in I_1} m_j$), and the others, which we denote by $\alpha_{M_1+1},...,\alpha_N$, being different from 1.

Write $A_2,...,A_n$ in the basis $e_1,...,e_N$:
$$A_j=\left [ a_{lm}^j \right ]_{l,m=1}^N.$$

For $t\in \C$ and $j\in I_1$ let ${\mathcal M}_j(t)$ be the $M_1\times M_1$ block of the matrix 
$$\left ( t-\sum_{l=2}^n \frac{\partial x_{1j} }{\partial x_l} \Bigg |_{(0,...,0)}x_l     \right )I-\sum_{l=2}^n x_lA_l$$
 formed by the first $M_1$ rows and the first $M_1$ columns. Here $\hat{x}=(x_2,\dots,x_n)$ is the point that appeared in Theorem \ref{limit_projections}. Consider the following polynomials in $t$
$${\mathcal S}_j(t)=det({\mathcal M}_j(t)). $$
For each $j$ this is a non-trivial polynomial of degree $M_1$ \ $\big ({\mathcal S}_j(t)\to \infty $ as $t\to \infty\big )$. Thus, there exists some positive number $b<c_1$ (the constant from \eqref{spectral_distance1}) such that these polynomials do not vanish in the punctured disk of radius $b$ centered at the origin:
\begin{equation}\label{determinant}
{\mathcal S}_j(t)\neq 0 \ \mbox{for} \ t\in \C, \ 0<|t|<b, \ j=1,...,k. 
\end{equation} 
Choose some $t_0$ satisfying $0<t_0 <b$. 

Next, we will show that there exists a positive constant ${\mathscr M}$ such that for $|w-1|=t_0 |z|$, where $z$ is close to zero, the following norm estimate holds:
\begin{equation}\label{norm_estimate}
\Bigg \Vert \Bigg ( wI - x_{1j}(z\hat{x})A_1-\sum_{s=2}^n zx_sA_s\Bigg )^{-1}  \Bigg \Vert \leq \frac{{\mathscr M}}{|z|}.\end{equation} 
The norm here is understood as the norm of an operator acting on $\C^N$.
 
Let $\eta\in \C^N, \ \parallel \eta \parallel =1$, and $\zeta \in \C^N$ satisfies 
\begin{equation}\label{system}
 \left ( wI-x_{1j}(z\hat{x})A_1-\sum_{l=2}^n zx_lA_l \right ) \zeta=\eta, \end{equation}
which is a system of $N$ equations in variables $\zeta_1,...,\zeta_N$, coordinates of $\zeta$ in the basis $e_1,...,e_N$. Since $w-1=t_0 ze^{i\theta}$ and $1-x_{1j}(z\hat{x})=z\left ( \sum_{s=2}^n \frac{\partial x_{1j}}{\partial x_s}|_{(0,...,0)}x_s\right )+O(|z|^2)$, the first $M_1$ equation of this system can be written as
\begin{equation}\label{subsystem}
z\bigg ( {\mathcal M}_j(t_0 e^{i\theta })+ O(z)I_{M_1} ) (^{M_1}\widehat{\zeta})\bigg)=(^{M_1}\widehat{\eta})-z{\mathcal N} \ ( \widehat{\zeta}^{N-M_1}),
\end{equation}
where $I_{M_1}$ is the $M_1\times M_1$ identity matrix, $(^{M_1}\widehat{\zeta})=(\zeta_1,...,\zeta_{M_1})$ and  $(^{M_1}\widehat{\eta})=(\eta_1,...,\eta_{M_1}), \ (\widehat{\zeta}^{N-M_1})=(\zeta_{M_1+1},...,\zeta_N)$, and ${\mathcal N}$ is the $M_1\times (N-M_1)$ part of the matrix $x_2A_2+...+x_nA_n$ consisting of the entrees in the first $M_1$ rows and the last $N-M_1$ columns.

Consider \eqref{subsystem} as a system of equations in $\zeta_1,...,\zeta_{M_1}$. Then \eqref{determinant} implies that for sufficiently small $z$ the main determinant of this system does not vanish (it is equal to $z^{M_1}{\mathcal S}_j(t_0 e^{i\theta})+O(|z|^{M_1+1})$). Now, Cramer's rule implies that $\zeta_1,...,\zeta_{M_1}$ are expressed in terms of $\zeta_{M_1+1}$, ..., $\zeta_N $ and $\eta_1,...,\eta_{M_1}$ in the following way:
\begin{equation}\label{zeta_i} 
\zeta_i=\left [ \frac{1}{z}\sum_{l=1}^{M_1}  \frac{D_{il}(z)}{{\mathcal S}_j(t_0 e^{i\theta})} \eta_l\right ] + \sum_{l=M+1}^N \frac{\psi_{il}(z)}{{\mathcal S}_j(t_0 e^{i\theta})}\zeta_l,\ i=1,...,k_1.
\end{equation}
where $D_{il}(z)$ and $\psi_{il}(z)$ are uniformly bounded analytic functions in a small punctured neighborhood of the origin (and, therefore, analytically extendable to the origin).
Substitute these expressions of $\zeta_i, \ i=1,...,M_1$ in the remaining $N-M_1$ equations of the system \eqref{system}. We obtain a system in the following form:
\begin{eqnarray}
\left [(1+t_0 z e^{i\theta})I_{N-M_1}-x_{1j}(z\hat{x})\Lambda - {\mathcal L}(z)\right ] (\widehat{\zeta}^{N-M_1}) \nonumber \\=(\widehat{\eta}^{N-M_1})-{\mathcal T}(z)(^{M_1}\widehat{\eta}) \label{remaining_system}, 
\end{eqnarray}
where $I_{N-M_1}$ is the $(N-M_1)\times (N-M_1)$ identity matrix, $\Lambda$ is the $(N-M_1)\times (N-M_1)$ diagonal matrix with $\alpha_{M_1+1},...,\alpha_N$ on the main diagonal (they are coming from $A_1$), ${\mathcal L}(z)$ is an $(N-M_1)\times (N-M_1)$ matrix whose all entries are of the order of $z$, and ${\mathcal T}(z)$ is the following $(N-M_1)\times M_1$ matrix with entrees
$$ \left ( {\mathcal T}(z) \right )_{im}=\frac{1}{{\mathcal S}_j(t_0 e^{i\theta})}\sum_{r=2}^nx_r\sum_{l=1}^{M_1} a_{il}^rD_{lm}(z),$$ 
that is the product of the lower left  $(N-M_1)\times M_1$  block of the matrix $\sum_{r=2}^nx_rA_r$ and the matrix $D(z)=\left [\frac{ D_{lm}(z)}{{\mathcal S}_j(t_0 e^{i\theta})} \right ]$.

Since for $z$ close to zero, $x_{1j}(z\hat{x})$ is close to one, and since $\alpha_{M+1},...,\alpha_N$ are not equal to one, it follows that the main determinant of \eqref{remaining_system} considered as a system in $\zeta_{M_1+1},...,\zeta_N$ stays away from 0 as $z\to 0$. Since for all $j=1,...,N$, \ $|\eta_J|\leq 1$, this implies that $\zeta_{M_1+1},...,\zeta_N$ stay bounded as $z\to 0$. In summary: there exists a constant $\tilde{{\mathscr M}}$ such that 
\begin{equation}\label{estimate}|\zeta_j|\leq \frac{\tilde{{\mathscr M}}}{|z|} \ j=1,...,M_1, \  |\zeta_j|\leq \tilde{{\mathscr M}}, \  j=M_1+1,...,N.
\end{equation} 
This yeilds 
\begin{equation}\label{estimate1}
\parallel \zeta \parallel \leq \frac{{\mathscr M}_1}{|z|}
\end{equation} for some constant ${\mathscr M}_1$ independent of $z$ and $\eta$, which proves \eqref{norm_estimate}. 

Now, \eqref{norm_estimate} implies
\begin{align} 
&\Bigg \Vert P_j(y_(z\hat{x})) \Bigg \Vert = \Bigg \Vert \frac{1}{2\pi i}\int_{|w-1|=t_0 |z|} \bigg(wI - x_{1j}(z\hat{x})A_1-\sum_{l=2}^n zx_lA_l\bigg)^{-1}dw  \Bigg \Vert  \nonumber \\
& \label{final estimate} \\
&\leq \frac{1}{2\pi }\left (\max_{	|w-1|=t_0 |z|} \Bigg \Vert \bigg(wI - x_{1j}(z\hat{x})A_1-\sum_{l=2}^n zx_lA_l\bigg)^{-1}\Bigg \Vert  \right ) 2\pi t_0 |z| \leq {\mathscr M}_1t_0. \nonumber
\end{align}

Since, the integral 
$$ \int_{|w-1|=t_0 |z|} \bigg(wI - x_{1j}(z\hat{x})A_1-\sum_{l=2}^n zx_lA_l\bigg)^{-1}dw $$
is an analytic function in $z$ in a small punctured disk centered at the origin, Riemann's theorem implies that it is analytically extendable to $z=0$. This completes the proof of Theorem \ref{limit_projections} for $\lambda \neq  0$.  
\vspace{.2cm}

2. $\lambda =0$.

The details of the proof in this case are very similar to those when $\lambda\neq 0$. The only differences are the following. If $\hat{x}$ is sufficiently close to 0, then close to 0 spectral points of the matrix $A_1+x_2A_2+\cdots +x_nA_n-x_{n+1,j}(\hat{x})I$ are  $x_{n+1,i}(\hat{x})-x_{n+1,j}(\hat{x}), \ i\in I_0$ (here we again write $x_{n+1,j}$ instead $x_{n+_1,0,j}$). Now, condition $\tilde{b})$ implies that there exists $c >0$ such that for sufficiently small $z$
$$|x_{n+1,i}(z\hat{x})-x_{n+1,j}(z\hat{x})|>c |z|, \ i,j\in I_0, \ i\neq j. $$
Further, we choose a basis $e_1,\dots,e_N$ which is an eigenbasis for $A_1$ so that
eigenvectors $e_1,\dots,e_{M_0}$ are 0-eigenvectors, and $e_{M_0+1},\dots,e_N$ are eigenvectors with non-trivial eigenvalues, and write $A_2,\dots,A_n$ in this basis. The matrix ${\mathcal M}_j(t)$ in this case is defined as $M_0\times M_0$ block of the matrix
$$\Bigg (t+ \sum_{l+2}^n \frac{\partial x_{n+1,j}}{\partial x_l} x_l \Bigg )I-\sum_{l=2}^n x_lA_l$$
of the first $M_0$ rows and the first $M_0$ columns. Again, there exists $b>0$ such that the polynomial ${\mathcal S}_j(t)=det({\mathcal M}_j(t))$ does not vanish in the punctured  disk 
of radius $b$ centered at the origin and we pick $0<t_0<b$. The remaining details of the proof are practically identical to those in the case $\lambda\neq 0$ (of course, the integral \eqref{projection} in \eqref{final estimate} is replaced with the integral \eqref{projection for 0}).

The proof is complete.


\vspace{.3cm}

\textbf{Remark 1}. We would like to note that the above proof holds when $A_1$ is just a diagonal matrix, not necessarily normal. 

\vspace{.2cm}

\textbf{Remark 2}. In general, limit projections depend on the choice of $\hat{x}$.

\vspace{.2cm} 

\textbf{Remark 3}. It is easily seen that a similar proof holds when a point approaches $(1,0,...,0)$ not along a line, but along a smooth curve $\gamma(z)\in \C^{n-1}$ which is not tangent to any component of $\hat{{\mathcal U}}$ at the origin.
\vspace{.3cm}

\section{\textbf{Relations between component projections ${\mathcal P}_{j,\lambda}$. \\ Proofs of Theorems \ref{first two moments} and \ref{A^2}}\label{relations1}}

\vspace{.2cm}

Each projection ${\mathscr P}_\lambda$ in the spectral resolution \eqref{A_1 decomposition} is, of course, represented by the integral
$${\mathscr P}_\lambda= \frac{1}{2\pi i}\int_{\gamma_\lambda} (w-A_1)^{-1}dw, $$
where $\gamma_\lambda$ is a contour which separates $\lambda$ from the rest of the spectrum of $A_1$.

\vspace{.2cm}

\begin{proposition}\label{independent_projections}
The limit projections $\Po_{j,\lambda}$ from Theorem \ref{limit_projections} satisfy the following relations:
\begin{eqnarray}
&{\mathcal P}_{i,\lambda}{\mathcal P}_{j,\lambda}=0 \ \mbox{if} \ i\neq j, \ i,j\in I_\lambda. \label{product_projections} \\
&\sum_{j\in I_\lambda} {\mathcal P}_{j,\lambda}={\mathscr P}_\lambda \label{sum_of_projections}
\end{eqnarray}
\end{proposition} 

\begin{proof}
The proofs for $\lambda\neq 0$ and $\lambda=0$ are very similar, so we will give it for $\lambda\neq 0$. 

Fix $\lambda\in \sigma(A_1), \ \lambda\neq 0$ and $j\in I_\lambda$. Let $\hat{x}\in {\mathcal O}$ - the neighborhood from the proof of Theorem \ref{limit_projections}. Equation \eqref{spectral_distance1} gives an estimate for $\mu_{ij}(\hat{x}), \ i\in I_\lambda$, the eigenvalues close to 1 of the operator $A(y_{j,\lambda}(\hat{x}))=x_{1,\lambda,j}(\hat{x})A_1+\cdots +x_nA_n$.  Recall that $\mu_{ij}(\hat{x})$ are the reciprocals of the roots of equation \eqref{root equation} which in the proof of Theorem \ref{limit_projections} were denoted by $\tau_{i,j}(1)$. In the proof of Theorem \ref{limit_projections} we supresed the dependance of these roots on $\lambda$, assuming that $\lambda=1$, but here it is appropriate to call them $\tau_{i,j,\lambda}(z)$. Recall that   $y_{j,\lambda}(\hat{x})$ was defined by \eqref{point on j component}.  Write
$$\tau_{i,j,\lambda}(\hat{x})= \tau_{i,j}(1), \ \tilde{y}_{i,j,\lambda}(\hat{x})=\tau_{i,j,\lambda}(\hat{x})y_{j,\lambda}(\hat{x}). $$

Of course, 
$$\tilde{y}_{j,j,\lambda}(\hat{x})=y_{j,\lambda}(\hat{x})$$. 
It was mentioned above that since $\tilde{R}_i$ is an irreducible polynomial, every regular point of $\{\tilde{R}_i=0\}$ has multiplicity 1, so $\tau_{i,j,\lambda}(\hat{x})$ is a bounded analytic function of $\hat{x}$ in a punctured neighborhood of the origin, and, therefore, can be extended to the origin to become an analytic function in the whole neighborhood satisfying 
$$ \tau_{i,j,\lambda}(0)=1.$$

It is also easy to compute that for $m=2,\dots,n$
\begin{equation}\label{tau_derivative}
\frac{\partial \tau_{i,j,\lambda}}{\partial x_m}(\hat{x})=\frac{ \tau_{i,j,\lambda}(\hat{x})\big(\frac{\partial x_{1,\lambda,i}}{\partial x_m}(\hat{x})-\frac{\partial x_{1,\lambda,j}}{\partial x_m}(\hat{x})\big)}{x_{1,\lambda,j}(\hat{x})-x_m\frac{\partial x_{1,\lambda,i}}{\partial x_m}(\hat{x})},	
\end{equation}
and, therefore,  \eqref{tau_derivative}implies
\begin{equation}\label{tau_derivative_at_0}
\frac{\partial \tau_{j,\lambda,i}}{\partial x_m}(0)=\frac{\frac{\partial x_{1i}}{\partial x_m}(0)-\frac{\partial x_{1j}}{\partial x_m}(0)}{\lambda}	.
\end{equation}

Further, the projections $P_{i,\lambda}(\tilde{y}_{j,j,\lambda}(\hat{x}))$ and  $ P_{j,\lambda}(\tilde{y}_{i,j,\lambda}(\hat{x}))$ (defined by the integral  \eqref{projection} ) satisfy
$$P_{i,\lambda}(\tilde{y}_{i,j,\lambda}(\hat{x})) P_{j,\lambda}(\tilde{y}_{j,j,\lambda}(\hat{x}))=0, \ i\neq j. $$

Since $\tau_{i,j,\lambda}(\hat{x})\hat{x} \to 0$ as $\hat{x}\to 0$,  Theorem \ref{limit_projections} implies   $P_{j,\lambda}(\tilde{y}_{j,j,\lambda}(\hat{x}))\to {\mathcal P}_{j,\lambda}$  and  $ P_{i,\lambda}(\tilde{y}_{i,j,\lambda}(\hat{x}))\to {\mathcal P}_{i,\lambda}$ as $\hat{x}\to 0$. This proves \eqref{product_projections}.

To prove \eqref{sum_of_projections} we fix a contour $\gamma_\lambda$ that separates $\lambda$ from the rest of the spectrum of $A_1$. Suppose that $x$ is close to $(1/\lambda,0,\dots,0)$. Consider the matrix
$$A(x)=x_1A_1+\cdots + x_nA_n.$$

Then $1\in \sigma(A_1/\lambda)$ and ${\mathscr P}_\lambda$ is the projection on the 1-eigenspace of $A_1/\lambda$. Also, $\gamma_\lambda/\lambda$ separates 1 from the rest of the spectrum of $A_1/\lambda$.

Of course, $A(x)\to A_1/\lambda$ as $\hat{x}\to 0, \ \mbox{and} \ x_1\to 1/\lambda$. If $|x_1-1/\lambda|+|x_2|+\cdots+|x_n|$ is small enough, there are no spectral points of $A_1/\lambda$ on $\gamma_\lambda/\lambda$, and, therefore,
$$\frac{1}{2\pi i}\int_{\gamma_\lambda/\lambda} (w-A(x))^{-1}dw \to \frac{1}{2\pi i }\int_{\gamma_\lambda /\lambda} \left (w-\frac{1}{\lambda}A_1\right )^{-1}dw={\mathscr P}_\lambda, $$  
as ${x}\to (1/\lambda,0,\dots,0).$

If $x_1=x_{1,\lambda,j}(\hat{x})$, then the poles of $(w-A(x))^{-1}$ which lie inside $\gamma_\lambda/\lambda$ are the reciprocals of $\tau_{i,j,\lambda}(\hat{x}), \ i \in I_\lambda$. Now, by the residue theorem
$$\frac{1}{2\pi i}\int_{\gamma_\lambda} (w-A(x))^{-1}dw=\sum_{i\in I_\lambda} P_{i,\lambda}( \tilde{y}_{i,j,\lambda}(\hat{x}))\to \sum_{i\i I_\lambda} {\mathcal P}_{i,\lambda},$$
and we are done.

\end{proof}

\vspace{.3cm}

We now concentrate on the case when $n=2$. Of course, in this case $\hat{{\mathcal U}}=\{0\}$, and the limit projections are the same for all $x_2$ (which is $\hat{x}$ in this situation).
\begin{proposition}\label{relations_between_projections}
Let $n=2, \ \lambda\in\sigma(A_1)$. Assume that conditions a) and b), or $\tilde{a})$ and $\tilde{b})$ are satisfied at $\lambda$ when $\lambda \neq 0$ or $\lambda=0$ respectively. For $i,j\in I_\lambda$ we have	
\begin{eqnarray}
&{\mathcal P}_{j,\lambda}A_2{\mathcal P}_{i,\lambda}=0, \ i\neq j \label{moment1_different}
\end{eqnarray}	
If there exists a neighborhood ${\mathcal O}$ of the origin, such that at every regular point of $\{\tilde{R}_j=0\}\cap {\mathcal O}$, the range of $P_{j,\lambda}(x_2)$ consists of 1-eigenvectors for $A(y_{j,\lambda}((x_2))=x_{1,\lambda,j}(x_2)A_1+x_2A_2$, if $\lambda \neq 0$, and , 0-eigenvectors of $A(y_{j,0}((x_2))=A_1+x_2A_2-x_{3,0,j}(x_2)I$ when $\lambda=0$, or equivalently,
\small\begin{eqnarray}  
&\Big(A(y_{j,\lambda}(x_2))-I\Big)P_{j,\lambda}(x_2)=P_{j,\lambda}(x_2)\Big(A(y_{j,\lambda}(x_2))-I\Big)=0, \ \lambda\neq 0, \label{alleigenvectors1}\\
&A(y_{j,0}((x_2))P_{j,0}(x_2)=P_{j,0}(x_2)A(y_{j,0}(x_2))=0	, \ \lambda=0 , \label{alleigenvectors2}
\end{eqnarray}\normalsize
 
then 
\small \begin{eqnarray}
&{\mathcal P}_{j,\lambda}^\prime \bigg(x_{1,\lambda,j}^\prime	(0)A_1+A_2 \bigg ){\mathcal P}_{j,\lambda}={\mathcal P}_{j,\lambda}\bigg ( x_{1,\lambda,j}^\prime	(0)A_1+A_2 \bigg ){\mathcal P}_{j,\lambda}^\prime \nonumber \\
&=-\frac{x_{1,\lambda,j}^{\prime \prime}(0)}{2}{\mathcal P}_{j,\lambda},  \ \lambda \neq 0, \label{relation1} \\
&\Po_{j,0}^\prime \Big(A_2-x_{3,,0,j}^\prime(0)I\Big)\Po_{j,0}=\Po_{j,0}\Big(A_2-x_{3,0,j}^\prime(0)I\Big) \nonumber \\
&= \frac{x_{3,0,j}^{\prime \prime}(0)}{2}\Po_{j,0}, \ \lambda=0 \label{relation2}\\
 &{\mathcal P}_{j,\lambda}^\prime \bigg ( x_{1,\lambda,j}^\prime(0)A_1+A_2\bigg ){\mathcal P}_{i,\lambda}={\mathcal P}_{j,\lambda}\bigg ( x_{1,\lambda,j}^\prime(0)A_1+A_2\bigg ){\mathcal P}_{i,\lambda}^\prime \nonumber \\
 &=0,  \ i\neq j, \ \lambda \neq 0,	 \label{relation3}\\
 &\Po_{j,0}^\prime \Big(A_2-x_{3,0,j}^\prime(0)I\Big)\Po_{i,0}=\Po_{j,0}\Big(A_2-x_{3,0,j}^\prime(0)I\Big)\Po_{i,0}^\prime \nonumber \\
 & =0, \ i\neq j, \ \lambda=0. \label{relation4}
\end{eqnarray}\normalsize

where ${\mathcal P}_{j,\lambda}^\prime=\lim_{x_2 \to 0} \frac{dP_{j,\lambda}(x_2)}{dx_2}, \ j\in I_\lambda.$

\end{proposition}

\begin{proof}
\vspace{.2cm}

\noindent I.  $\lambda\neq 0$.

We have for every $x_2$ close to 0:
\begin{equation}\label{identity} 
P_{j, \lambda}(x_2)A(x_{1,\lambda,j}(x_2),x_2)=A(x_{1,\lambda,j}(x_2),x_2)P_{j,\lambda}(x_2).
\end{equation}	
By Theorem \ref{limit_projections} both sides of this equality are analytic functions of $x_2$ in a neighborhood of the origin. Differentiate with respect to $x_2$:
\begin{align*}
&\frac{dP_{j,\lambda}(x_2)}{dx_2}A(x_{1,\lambda,j}(x_2),x_2)+P_{j,\lambda}(x_2)(x_{1,\lambda,j}^\prime (x_2)A_1+A_2) \\
&=A(x_{1,\lambda,j}(x_2),x_2) \frac{dP_{j,\lambda}(x_2)}{dx_2}+(x_{1,\lambda,j}^\prime (x_2)A_1+A_2)P_{j,\lambda}(x_2). 	
\end{align*}
Passing to the limit as $x_2 \to 0$ and using   ${\mathcal P}_{j,\lambda}A_1=A_1{\mathcal P}_{j,\lambda}=\lambda{\mathcal P}_{j,\lambda}$, we obtain
$$\frac{1}{\lambda}{\mathcal P}_{j,\lambda}^\prime A_1+\lambda x_{1,\lambda,j}^\prime(0){\mathcal P}_{j,\lambda}+{\mathcal P}_{j,\lambda}A_2=\frac{1}{\lambda}A_1{\mathcal P}_{j,\lambda}^\prime +\lambda x_{1,\lambda,j}^\prime(0){\mathcal P}_{j,\lambda}+A_2{\mathcal P}_{j,\lambda},$$
%
so that

$$\frac{1}{\lambda}{\mathcal P}_{j,\lambda}^\prime A_1+{\mathcal P}_{j,\lambda}A_2=\frac{1}{\lambda}A_1{\mathcal P}_{j,\lambda}^\prime +A_2{\mathcal P}_{j,\lambda}. $$
Multiply the last equality from the left by ${\mathcal P}_{i,\lambda}$ with $i\neq j, \ i\in I_\lambda$. Since ${\mathcal P}_{i,\lambda}A_1=A_1{\mathcal P}_{i,\lambda}=\lambda{\mathcal P}_{i,\lambda}$, we obtain

$$\frac{1}{\lambda} {\mathcal P}_{i,\lambda}{\mathcal P}_{j,\lambda}^\prime A_1={\mathcal P}_{i,\lambda}{\mathcal P}_{j,\lambda}^\prime+{\mathcal P}_{i,\lambda}A_2{\mathcal P}_{j,\lambda}$$
yeilding
$${\mathcal P}_{i,\lambda}{\mathcal P}_{j,\lambda}^\prime \Bigg(\frac{1}{\lambda}A_1-I\Bigg)= {\mathcal P}_{i,\lambda}A_2{\mathcal P}_{j,\lambda}.$$
Multiply the last equality by ${\mathcal P}_{j,\lambda}$ from the right. Since $$\Bigg(\frac{1}{\lambda}A_1-I\Bigg){\mathcal P}_{j,\lambda}=0$$, and ${\mathcal P}_{j,\lambda}^2={\mathcal P}_{j,\lambda}$, we conclude that
$$ {\mathcal P}_{i,\lambda}A_2{\mathcal P}_{j,\lambda}=0,$$ 
and \eqref{moment1_different} is proved for $\lambda\neq 0$..

If \eqref{alleigenvectors1} holds, then
\begin{eqnarray}
&P_{j,\lambda}(x_2)A(x_{1,\lambda,j}(x_2),x_2)=P_{j,\lambda}(x_2) \label{next1}\\
& A(x_{1,\lambda,j}(x_2),x_2)P_{j,\lambda}(x_2)= P_{j,\lambda}(x_2).	\label{next2}
\end{eqnarray}

Differentiating twice \eqref{next1} and \eqref{next2} with respect to $x_2$  and passing to the limit as $x_2 \to 0$ yields



\begin{eqnarray*}
&\frac{1}{\lambda}{\mathcal P}_{j,\lambda} ^{\prime \prime} A_1+2{\mathcal P}_{j,\lambda}^\prime(x_{1,\lambda,j}^\prime (0)A_1+A_2)+x_{1,\lambda,j}^{\prime \prime}(0){\mathcal P}_{j,\lambda}A_1={\mathcal P}_{j,\lambda}^{\prime \prime} , \\
&x_{1,\lambda,j}^{\prime \prime}(0)A_1{\mathcal P}_{j,\lambda}+2(x_{1,\lambda,j}^\prime(0)A_1+A_2){\mathcal P}_{j,\lambda}^\prime+\frac{1}{\lambda}A_1{\mathcal P}_j^{\prime \prime}={\mathcal P}_{j,\lambda}^{\prime \prime},
	\end{eqnarray*}
where ${\mathcal P}^{\prime \prime}_{j,\lambda}=\lim_{x_2\to 0} \frac{d^2 P_{j,\lambda}(x_2)}{dx_2^2}$. Moving the righthand side of the last two equations to the left gives

\small\begin{eqnarray}
&{\mathcal P}_{j,\lambda} ^{\prime \prime}( \frac{1}{\lambda}A_1-I)+2{\mathcal P}_{j,\lambda}^\prime(x_{1,\lambda,j}^\prime (0)A_1+A_2)+x_{1,\lambda,j}^{\prime \prime}(0){\mathcal P}_{j,\lambda}A_1=0, \label{next3}\\
&x_{1,\lambda,j}^{\prime \prime}(0)A_1{\mathcal P}_{j,\lambda}+2(x_{1,\lambda,jj}^\prime(0)A_1+A_2){\mathcal P}_{j,\lambda}^\prime+(\frac{1}{\lambda}A_1-I){\mathcal P}_{j,\lambda}^{\prime \prime}=0.	\label{next4}
\end{eqnarray}\normalsize

Multiply \eqref{next3} from the right by ${\mathcal P}_{j,\lambda}$ and  by ${\mathcal P}_{i,\lambda}, \ i\neq j $.
 Again, since 
 \begin{eqnarray*}
 &\Bigg(\frac{1}{\lambda}A_1-I\Bigg)\Po_{l,\lambda}=\Po_{l,\lambda}\Bigg(\frac{1}{\lambda}A_1-I\Bigg)=0, \ l\in I_\lambda, \\
 &	\Po_{l,\lambda}A_1=A_1\Po_{l,\lambda}=\lambda \Po_{l,\lambda}, \ l\in I_\lambda \\
 & \Po_{i,\lambda}\Po_{j,\lambda}=\Po_{j,\lambda}\Po_{i,\lambda}=0, \ i,j\in I_\lambda, \ i\neq j,
 \end{eqnarray*}   
 we obtain

\begin{eqnarray*}
&{\mathcal P}_{j,\lambda}^\prime \Big(x_{1,\lambda,j}^\prime	(0)A_1+A_2\Big){\mathcal P}_{j,\lambda}=-\frac{x_{1,\lambda,j}^{\prime \prime}(0)}{2}{\mathcal P}_{j,\lambda} \\
& {\mathcal P}_{j,\lambda}^\prime \Big(x_{1,\lambda,j}^\prime(0)A_1+A_2\Big){\mathcal P}_{i,\lambda}=0, \ i,j\in I_\lambda, \ i\neq j,
\end{eqnarray*}
Similarly, multiplication of \eqref{next4} from the left by $\Po_{j,\lambda}$ and by $\Po_{i,\lambda}$ results in

\begin{eqnarray*}
& {\mathcal P}_{j,\lambda} \Big(x_{1,\lambda,j}^\prime	(0)A_1+A_2\Big){\mathcal P}_{j,\lambda}^\prime=-\frac{x_{1,\lambda,j}^{\prime \prime}(0)}{2}{\mathcal P}_{j,\lambda} \\
& {\mathcal P}_{i,\lambda} \Big(x_{1,\lambda,j}^\prime(0)A_1+A_2\Big){\mathcal P}_{j,\lambda}^\prime=0, \ i,j\in I_\lambda, \ i\neq j,
\end{eqnarray*}
which finishes the proof for $\lambda \neq 0$.

\noindent II $\lambda=0$. 

The proof in this case is similar and goes along the following lines.
\begin{eqnarray*}
&\Big ( A_1+x_2A_2-x_{3,o,j}(x_2)I\Big )P_{j,0}(x_2)= P_{j,0}(x_2)\Big ( A_1+x_2A_2-x_{3,0,j}(x_2)I\Big ),\\
&\Big( A_2-x_{3,0,j}^\prime (0)I\Big)\Po_{j,0}+A_1\Po_{j,0}^\prime=\Po_{j,0}^\prime A_1+\Po_{j,0}\Big(A_2-x_{3,0,j}^\prime(0)I\Big).
\end{eqnarray*} 
Multiply the last relation from the right by $\Po_{i,0}$ and use 
$$\Po_{j,0}\Po_{i,0}=\Po_{i,0}\Po_{j,0}=0, \ i\neq j, \ A_1\Po_{l,0}=\Po_{l,0}A_1=0, \ l\in I_0.$$ 
Since 
$$P_{j,0 }(x_2)P_{i,0}(x_2)\equiv 0, \ i\neq j, \ \Po_{j,0}^\prime \Po_{i,0}=-\Po_{j,0}\Po_{i,0}^\prime,$$ 
 we obtain
$$ 0=-A_1\Po_{j,0}\Po_{i,0}^\prime=A_1\Po_{j,0}^\prime \Po_{i,0}=\Po_{j,0}A_2\Po_{i,0}, $$
which is \eqref{moment1_different} for $\lambda=0$.

Finally, under the assumption of \eqref{alleigenvectors2}, the proof of \eqref{relation2} and \eqref{relation4} is obtained the same way as the one of \eqref{relation1} and \eqref{relation3} by twice differentiating 
\begin{eqnarray*}
&P_{j,0}(x_2)\Big ( A_1+x_2A_2\Big)=x_{3,,0,j}(x_2)P_{j,0}(x_2) \\
&\Big ( A_1+x_2A_2\Big)P_{j,0}(x_2)=x_{3,0,j}(x_2)P_{j,0}(x_2)	
\end{eqnarray*}
with respect to $x_2$, passing to the limit as $x_2\to 0$, and multiplying by $\Po_{j,0}$ and $\Po_{i,0}$.
 We are done.
\end{proof}
Obviously, \eqref{alleigenvectors1} and \eqref{alleigenvectors2} hold when every spectral component passing through $(1/\lambda,0)$ has multiplicity 1.


\vspace{.4cm}

\Large \textit{Proof of Theorem \ref{first two moments}}\normalsize

\vspace{.3cm}

I. $\lambda \neq 0$

Fix $j\in I_\lambda$. 
In the results of the previous section we were interested only in the roots of \eqref{root equation}
\begin{equation}\label{roots1} 
 \tilde{R}_i(\tau x_{1,\lambda,j}(x_2), \tau x_2)=0. 
 \end{equation}
which were close to 1. Now we will consider all the roots of this equation.  Observe that if $x_2$ is small enough, all these roots are close to the ratios $\lambda/\mu$ where $\mu\in \sigma(A_1)$. Moreover, since condition b) holds at every spectral point of $A_1$, the multiplicity of each of these roots is 1 (when $x_2$ is close to 0). We denote them by $\tau_{j,\lambda,i,\mu}(x_2), \ 1\leq i\leq s, \ $, where $\mu$ is defined by
\begin{equation}\label{tau limit}
\tau_{j,\lambda,i,\mu}(x_2) \to   \frac{\lambda}{\mu}, \ \mbox{as} \ x_2\to 0, 
\end{equation}     
which, of course, implies $i\in I_\mu$. It was mentioned above that $\tau_{j,\lambda,i,\mu}(x_2)$ are analytic functions of $x_2$ in a neghborhood of the origin. Obviously, $\tau_{j,\lambda,j,\lambda}(x_2)=1$.

Further, \eqref{projection} defined projections $P_{j,\lambda}(x_2)$. We define
\begin{equation}\label{spectral projection}
P_{j,\lambda,i,\mu}(x_2)=P_{i,\mu}(\tau_{j,\lambda,i,\mu} \ x_2).	
\end{equation}
It follows from Theorem \ref{limit_projections} and \eqref{spectral projection} that $P_{j,\lambda,i,\mu}(x_2)$ are analytic operator-valued functions of $x_2$ in a neighborhood of the origin, and
\begin{equation}\label{projection limit}
\lim_{x_2\to 0} P_{j,\lambda,i,\mu}(x_2)= \Po_{i,\mu}.	
\end{equation}
Also
$$\frac{d P_{j,\lambda,i,\mu}}{dx_2}(x_2)=\frac{d P_{i,\lambda}(\tau_{j,\lambda,i,\mu} \ x_2)}{dx_2}\bigg [\tau_{j,\lambda,i,\mu}(x_2)+x_2\frac{d\tau_{j,\lambda,i,\mu}(x_2)}{dx_2}\bigg ], $$
so that \eqref{tau_derivative_at_0} and \eqref{tau limit} imply that for $\lambda=\mu$ we have
\begin{equation}\label{P derivative at 0}
\frac{d P_{j,\lambda,i,\lambda}}{dx_2}(0)=\lim_{x_2\to 0}\frac{d P_{j,\lambda,i,\lambda}}{dx_2}(x_2)=\Po_{i,\lambda}^\prime(0).	
\end{equation}

Recall that projections \eqref{spectral projection} appeared in Proposition \ref{independent_projections} for $\lambda=\mu$. We renamed them here, since we need these projections corresponding to all roots of \eqref{roots1} and want to indicate to which reciprocal of a point in $\sigma(A_1)$ the corresponding point of the proper joint spectrum converges. Note that $\tau_{j,\lambda,j,\lambda}(x_2)=1$ implies $P_{j,\lambda,j,\lambda}(x_2)=P_{j,\lambda}(x_2)$. 


Since for $i\in I_\lambda$ components $\{\tilde{R}_i(x_1,x_2)=0\}$ have multiplicity 1 in the projective joint spectrum , the rank of $P_{j,\lambda,i,\mu}(x_2)$ is equal to 1 for every $i\in I_\lambda \cap I_\mu$. For other $i\notin I_\lambda$ the rank of $P_{j,\lambda,i,\mu}$ is equal to $m_i$ - the multiplicity of the spectral component $\{\tilde{R}_i=0\}$.

As it was mentioned above, for a rank 1 projection \eqref{alleigenvectors1} holds, and, therefore, if $\Delta$ is small, we have
$$\Big(x_{1,\lambda,j}(x_2+\Delta)A_1+(x_2 +\Delta)A_2-I\Big)P_{j,\lambda}(x_2+\Delta)=0.$$
We write this relation in the form
\begin{equation}\label{annihilation1}
\frac{1}{2\pi i}\int_\gamma \big(w-1\big)\bigg(w- x_{1,\lambda,j}(x_2+\Delta)A_1-(x_2 +\Delta)A_2\bigg)^{-1}dw=0,
\end{equation}
where, as above, $\gamma$ is a circle centered at 1 which separates 1 from all other eigenvalues
of $x_{1,\lambda,j}(x_2)A_1+x_2A_2$, and, therefore, the same is true for all $x_{1,\lambda,j}(x_2+\Delta)A_1+(x_2+\Delta)A_2$ for all sufficiently small $\Delta$. 

Let us write \eqref{annihilation1} as

\begin{eqnarray}
&0=\frac{1}{2\pi i}\int_\gamma \big (w-1\big )\Big (w-x_{1,\lambda,j}(x_2	)A_1-x_2A_2\Big)^{-1} \nonumber \\
&\times \left [ I-\left ( \big(x_{1,\lambda,j}^\prime(x_2)A_1+A_2 \big)\Delta+\left (\sum_{k=2}^\infty\frac{x_{1,\lambda,j}^{(k)}(x_2)}{k!}\Delta^k \right )A_1 \right )\right. \nonumber \\
& \times \Big (w-x_{1,\lambda,j}(x_2	)A_1-x_2A_2 \Big )^{-1}\bigg ]^{-1}dw \nonumber  \\
&=\frac{1}{2\pi i}\int_\gamma \big (w-1\big )\Big (w-x_{1,\lambda,j}(x_2)A_1-x_2A_2\Big )^{-1} \label{variation} \\
& \times \sum_{l=0}^\infty \left [\Bigg(\big( x_{1,\lambda,j}^\prime(x_2)A_1+A_2\big)\Delta+ \left (\sum_{k=2}^\infty\frac{x_{1,\lambda,j}^{(k)}(x_2)}{k!}\Delta^k \right )A_1  \Bigg) \right. \nonumber \\
&\times \Big (w-x_{1,\lambda,j}(x_2	)A_1-x_2A_2\Big )^{-1}\bigg ]^ldw. \nonumber
\end{eqnarray}

The right hand side of the last expression is an analytic function of $\Delta $ in a small neighborhood of $0$, so that all the derivatives at $\Delta=0$ vanish. We will need only the first and the second ones. Here are the corresponding relations.
\small\begin{eqnarray}
&\frac{1}{2\pi i}\int_\gamma \big (w-1\big )\Big (w-x_{1,\lambda,j}(x_2)A_1-x_2A_2\Big)^{-1}(x_{1,\lambda,j}^\prime (x_2)A_1+A_2)	\nonumber \\
&\times \Big (w-x_{1,\lambda,j}(x_2)A_1-x_2A_2\Big )^{-1}dw=0. \label{residue1}\\
&\frac{1}{2\pi i}\int_\gamma \big (w-1\big )\Big (w-x_{1,\lambda,j}(x_2)A_1-x_2A_2\Big )^{-1}\Bigg \{ \Bigg(( x_{1,\lambda,j}^\prime (x_2)A_1+A_2) \nonumber \\
&\times \Big(w-x_{1,\lambda,j}(x_2)A_1-x_2A_2\Big)^{-1}(x_{1,\lambda,j}^\prime (x_2)A_1+A_2)\Big(w-x_{1,\lambda,j}(x_2)A_1-x_2A_2\Big )^{-1}\Bigg ) \nonumber \\
& + \frac{x_{1,\lambda,j}^{\prime \prime}(x_2)}{2}A_1\Big(w-x_{1,\lambda,j}(x_2)A_1-x_2A_2\Big)^{-1}\Bigg \}dw=0. \label{residue2}
\end{eqnarray}\normalsize

We will now express relations \eqref{residue1} and \eqref{residue2} in terms of projections $P_{j,\lambda,i,\mu}(x_2)$.

The spectrum of the matrix $x_{1,\lambda,j}(x_2)A_1+x_2A_2$ consists of complex numbers $\mu_{j,\lambda,i,\nu}(x_2)$
 which are reciprocals of $\tau_{j,\lambda,i,\nu}(x_2), \ \nu\in \sigma(A_1)$ and, possibly, 0. 
 If $x_2$ is close to 0, the latter might occur only when $0\in \sigma(A_1)$. If 0 is an eigenvalue of $x_{1,\lambda,j}(x_2)A_1+x_2A_2$, we include it in the formulas below as $\mu_{j,\lambda,i,0}(x_2) $.

As mentiioned above, each eigenvalue $\mu_{j,\lambda,i,\nu}(x_2), \ i\in I_\lambda, \ \nu=\lambda$ is close to 1. These eigenvalues  tend to 1 as $x_2\to 0$ (and, of course, $\mu_{j,\lambda,j,\lambda}(x_2)=1$ for all $x_2$). All other $\mu_{j,\lambda,i,\nu}(x_2)$ stay away from 1 as $x_2\to 0$, that is for some positive constant $a$
\begin{equation}\label{estimate for mu}
|\mu_{j,\lambda,i,\nu}(x_2)-1|>a, \ \mbox{if} \  i\in I_\lambda \ \mbox{and} \ \nu \neq \lambda, \ \mbox{ or} \ i\notin I_\lambda. 
\end{equation}


The Jordan decomposition of the matrix $A(x_2)=x_{1,\lambda,j}(x_2)A_1+x_2A_2$ consists of Jordan blocks which have sizes between 
 1  and $m_i$ corresponding to eigenvalue $\mu_{j,\lambda,i,\nu}$ (we consider an eigenvector which is not in a Jordan cell of dimension higher than 1 as a Jordan cell of size 1), so for $w$ close to 1 we have
\begin{eqnarray}
\Big(w-x_{1j}(x_2)A_1-x_2A_2\Big)^{-1}=\frac{P_{j,\lambda,j,\lambda}(x_2) }{w-1}+
\sum_{i\in I_\lambda, i\neq j}\frac{P_{j,\lambda,i,\lambda}(x_2)}{w-\mu_{j,\lambda,i,\lambda}(x_2)} \nonumber \\
 \label{inverse}\\
+\sum_{\nu \in \sigma(A_1), \nu \neq \lambda}\sum_{i \in I_\nu} \sum_{r=0}^{m_i-1}\frac{\Big(\mu_{j,\lambda,i,\nu}(x_2)I-x_{1,\lambda,j}(x_2)A_1-x_2A_2\Big)^rP_{j,\lambda,i,\nu}(x_2)}{\Big(w-\mu_{j,\lambda,i,\nu}(x_2)\Big)^{r+1}}, \nonumber
\end{eqnarray}

The integrands of \eqref{residue1} and \eqref{residue2} are meromorphic functions in $w$ in a neighborhood of $w=1$. We use \eqref{inverse} to compute their residues at 1. A simple but tedious computation shows that
\begin{eqnarray}
&Res\Big|_{w=1}\Big[(w-1)\big(w-x_{1,\lambda,j}(x_2)A_1-x_2A_2\big)^{-1}(x_{1,\lambda,j}^\prime (x_2)A_1+A_2) \nonumber\\
&\times \big(w-x_{1,\lambda,j}(x_2)A_1-x_2A_2\big)^{-1}\Big ]	\label{first_residue} \\
&= P_{j,\lambda,j,\lambda}(x_2)(x_{1,\lambda,j}^\prime(x_2)A_1+A_2){\mathcal P}_{j,\lambda,j,\lambda}(x_2)=0 \nonumber
\end{eqnarray}
Passing to the limit as $x_2\to 0$ we obtain 
\begin{equation}\label{moment1}
\Po_{j,\lambda}(x_{1,\lambda,j}^\prime(0)A_1+A_2)\Po_{j,\lambda}=0, 
\end{equation}
which is the first relation claimed in Theorem \ref{first two moments}.

To simplify the computation of the residue of \eqref{residue2} at $w=1$ let us introduce the following operators
\begin{eqnarray*}
&{\mathscr R}(x_2)=x_{1,\lambda,j}^\prime(x_2)A_1+A_2, \\
& \\
&{\mathscr S}_1(x_2)=P_{j,\lambda,j,\lambda}(x_2){\mathscr R}(x_2)\left ( \displaystyle \sum_{i\in I_\lambda, i\neq j}\frac{P_{j,\lambda,i,\lambda}(x_2)}{1-\mu_{j,\lambda,i,\lambda}(x_2)}\right ){\mathscr R}(x_2)P_{j,\lambda,j,\lambda}(x_2), \\
& \\
&{\mathscr S}_2(x_2)=P_{j,\lambda,j,\lambda}(x_2){\mathscr R}(x_2) \\ 
&\times\Bigg (\displaystyle  \sum_{\nu \in \sigma(A_1), \nu \neq \lambda}\sum_{i \in I_\nu} \sum_{r=0}^{m_i-1}\frac{\Big(\mu_{j,\lambda,i,\nu}(x_2)I-x_{1,\lambda,j}(x_2)A_1-x_2A_2\Big)^rP_{j,\lambda,i,\nu}(x_2)}{\Big(1-\mu_{j,\lambda,i,\nu}(x_2)\Big)^{r+1}}\Bigg )  \\
&\times {\mathscr R}(x_2)P_{j,\lambda,j,\lambda}(x_2).
\end{eqnarray*}

A straightforward computation using \eqref{inverse} and \eqref{first_residue} shows that \eqref{residue2} turns into

\begin{equation}\label{second_residue} 
{\mathscr S}_1(x_2)+{\mathscr S}_2(x_2)+\frac{x_{1,\lambda,j}^{\prime \prime}(x_2)}{2}P_{j,\lambda,j,\lambda}(x_2)A_1P_{j,\lambda,j,\lambda}(x_2)=0.
\end{equation}

Our next step is finding the limits of each of terms of \eqref{second_residue} as $x_2\to 0$.  

If $i\in I_\lambda$, \ $\mu_{j,\lambda,i, \lambda}(x_2)\to 1$ as $x_2\to 0$, and condition b) implies that $\mu_{j,\lambda,i,\lambda}^\prime(0)\neq 0$, so we have

\begin{eqnarray*}
&{\mathscr S}_1(x_2)=P_{j,\lambda,j,\lambda}(x_2){\mathscr R}(x_2)\left (\displaystyle \sum	_{i\in I_\lambda, i\neq j}\frac{P_{j,\lambda,i,\lambda}(x_2)}{1-\mu_{j,\lambda,i,\lambda}(x_2)}\right ){\mathscr R}(x_2)P_{j,\lambda,j,\lambda}(x_2) \\
&=P_{j,\lambda,j,\lambda}(x_2){\mathscr R}(x_2)\left (\displaystyle \sum	_{i\in I_\lambda, i\neq j}\frac{P_{j,\lambda,i,\lambda}(x_2)}{-\mu_{j,\lambda,i,\lambda}^\prime(0)x_2+O(|x_2|^2)}\right ){\mathscr R}(x_2)P_{j,\lambda,j,\lambda}(x_2) \\
&=\left ( \displaystyle \sum	_{i\in I_\lambda, i\neq j}\frac{	P_{j,\lambda,j,\lambda}(x_2){\mathscr R}(x_2)P_{j,\lambda,i,\lambda}(x_2)}{x_2(-\mu_{j,\lambda,i,\lambda}^\prime(0)+o(1))} \right )  {\mathscr R}(x_2)P_{j,\lambda,j,\lambda}(x_2) \\
&=\left (\displaystyle \sum_{i\in I_\lambda, i\neq j}\frac{P_{j,\lambda,j,\lambda}(x_2)\Big({\mathscr R}(0)+(x_{1,\lambda,j}^{\prime \prime}(0)x_2+O(|x_2|^2))A_1\big)P_{j,\lambda,i,\lambda}(x_2)}{x_2(-\mu_{j,\lambda,i,\lambda}^\prime(0)+o(1))} \right ) \\
&\times {\mathscr R}(x_2)P_{j,\lambda,j,\lambda}(x_2).
\end{eqnarray*}

\vspace{.2cm}

Write 
$$\Phi (x_2)=P_{j,\lambda,j,\lambda}(x_2){\mathscr R}(0)P_{j,\lambda,i,\lambda}(x_2). $$

Since  ${\mathcal P}_{j,\lambda}A_1{\mathcal P}_{i,\lambda}=\lambda {\mathcal P}_{j,\lambda}{\mathcal P}_{i,\lambda}=0$ for $i,j\in I_\lambda, \ i\neq j$, relation \eqref{moment1_different} in  Proposition \ref{independent_projections} 
implies 
$$
\Phi(0)=0. 
$$
Also, \eqref{P derivative at 0} and \eqref{relation3} yield
$$\Phi^\prime(0)=0, $$

and, hence,
$$\Phi(x_2)=O(|x_2|^2), \ \mbox{as} \ x_2\to 0, $$ 
resulting in
\begin{equation}\label{S_1}
\lim_{x_2 \to 0}{\mathscr S}_1(x_2)=0.	
\end{equation}

To evaluate the limit of ${\mathscr S}_2(x_2)$ we observe that since $A_1$ is normal, Theorem \ref{limit_projections} amd \eqref{tau limit} impliy that for $\lambda\neq \nu, \ i\in I_\nu $ and 
$$((\mu_{j,\lambda,i,\nu}(x_2)I-x_{1,\lambda,j}(x_2)A_1-x_2A_2)^rP_{j,\lambda,i,\nu}(x_2) \to 0 \ \mbox{for $r>0$ as} \ x_2\to 0.$$
Since by \eqref{estimate for mu} we have 
$|1-\mu_{j,\lambda,i,\lambda}(x_2)|>a>0 $, and, since by Theorem \ref{limit_projections} all projections $P_{j,\lambda,i,\nu}(x_2)$ are bounded, all terms in the expression of ${\mathscr S}_3(x_2)$ with $r>0$ tend to 0 as $x_2\to 0$. This yields

\begin{equation}\label{S_3}
\lim_{x_2 \to 0}{\mathscr S}_2(x_2)={\mathcal P}_{j,\lambda}A_2 \left ( \displaystyle \sum_{\nu\in \sigma(A_1), \nu\neq \lambda} \sum_{i\in I_\nu} \frac{{\mathcal P}_{i,\nu}}{1-\mu_{j,\lambda,i,\nu}(0)}\right )A_2{\mathcal P}_{j,\lambda}.	
\end{equation}
Equations \eqref{sum_of_projections} in Proposition \ref{independent_projections}, \eqref{tau limit}, and \eqref{S_3} now give
\begin{equation} \label{almost there}
\lim_{x_2 \to 0}{\mathscr S}_2(x_2)={\mathcal P}_{j,\lambda}A_2 \left ( \displaystyle \sum_{\nu\in \sigma(A_1), \nu\neq \lambda} \frac{\lambda {\mathscr P}_\nu}{\lambda -\nu}\right )A_2\Po_{j,\lambda}=\lambda \Po_{j,\lambda}A_2T_\lambda A_2\Po_{j,\lambda}.\end{equation} 

Finally, passing to the limit in \eqref{second_residue} as $x_2\to 0$, \eqref{S_1}, and \eqref{almost there} along with $\Po_{j,\lambda}A_1=A_1\Po_{j,\lambda}=\lambda \Po_{j,\lambda}$, \ $\Po_{j,\lambda}^2=\Po_{j,\lambda}$, result in

$$\Po_{j,\lambda}A_2T_\lambda A_2\Po_{j,\lambda}+\frac{x_{1,\lambda,j}^{\prime \prime}(0)}{2}\Po_{j,\lambda}=0,
$$
which finishes the proof of Theorem \ref{first two moments} for $\lambda \neq 0$.

\vspace{.2cm}

\noindent II. $\lambda=0$

In this case \eqref{projection for 0} gives the expression for $P_{j,0}(x_2)$. 

The spectrum of $A_1+x_2A_2$ consists of the roots of
$$R_i(1,x_2,z)=\tilde{R}_j(x_2,z)=0. $$
If $x_2$ is close to 0, we denote these roots by $x_{3,\nu,i}(x_2)$, where $\nu\in \sigma(A_1)$ is the eigenvalue of $A_1$ to which $x_{3,\nu,i}(x_2)$ converges as $x_2\to 0$, and $i\in I_\nu$. Each $x_{3,\nu,i}$ has multiplicity $m_i$, and, in particular, this multiplicity is equal to 1 for $i\in I_0$.



Fix $ j\in I_0$ and $x_2$ close to 0. Similar to what we had in the case $\lambda \neq 0$, here
$$\Big(A_1+(x_2+\Delta) A_2- x_{3,0,j}(x_2+\Delta )I\Big) P_{j,0}(x_2)=0,$$
if $\Delta$ is small enough. An analog of \eqref{annihilation1} in this setting is
\begin{equation}\label{annihilation5}
\frac{1}{2\pi i}\int_\gamma w\Big((w+x_{3,0,j}(x_2+\Delta))I-A_1-(x_2 +\Delta)A_2\Big)^{-1}dw=0,
\end{equation}
and, as above, $\gamma$ separates 0 from the rest of the spectrum of $A_1+x_2A_2-x_{3,0,j}(x_2)I$. Condition $\tilde{b})$ implies that there is $a>0$ independent of $x_2$ such that $\gamma$ can be taken to be circle centered at the origin of radius $a|x_2|$. 

Following a similar line of presentation as for $\lambda\neq 0$, we write \eqref{annihilation5} as a power series in $\Delta$:
\begin{eqnarray}
&0=\frac{1}{2\pi i}\int_\gamma w\Big((w+x_{3,0,j}(x_2)I-A_1-x_2A_2\Big)^{-1} \nonumber \\
&\times \left [ I-\left \{ \big(A_2-x_{3,0,j}^\prime(x_2) I\big)\Delta-\left (\displaystyle \sum_{k=2}^\infty\frac{x_{3,0,j}^{(k)}(x_2)}{k!}\Delta^k\right )I \right \}\right. \nonumber \\
& \times \Big((w+x_{3,0,j}(x_2)I-A_1-x_2A_2\Big)^{-1}\bigg ]^{-1}dw \nonumber  \\
&=\frac{1}{2\pi i}\int_\gamma  w\Big((w+x_{3,0,j}(x_2)I-A_1-x_2A_2\Big)^{-1} \label{variation2} \\
& \times \displaystyle \sum_{l=0}^\infty \left [\left \{ \big(A_2-x_{3,0,j}^\prime(x_2) I\big)\Delta-\left (\sum_{k=2}^\infty\frac{x_{3,0,j}^{(k)}(x_2)}{k!}\Delta^k\right )I \right \} \right. \nonumber \\
&\times \Big((w+x_{3,0,j}(x_2)I-A_1-x_2A_2\Big)^{-1}\bigg ]^ldw. \nonumber
\end{eqnarray}
and equate the coefficients for powers of $\Delta$ to 0. Again, we are interested in the the coefficients for $\Delta$ and $\Delta^2$.

In this case the inverse is given by:
\large \begin{eqnarray}
&\Big((w+x_{3,0,j}(x_2)I-A_1-x_2A_2\Big)^{-1} \nonumber\\
&= \frac{P_{j,0}(x_2)}{w}+\displaystyle \sum_{i\in I_0, i\neq j} \frac{P_{i,0}(x_2)}{w+x_{3,0,j}(x_2)-x_{3,0,i}(x_2)} \nonumber\\
& \label{inverse for 0} \\
&+\displaystyle \sum_{\nu\in \sigma(A_1), \nu\neq 0} \sum_{i\in I_\nu} \sum_{l=0}^{m_i-1}\frac{\big(A_1+x_2A_2 -x_{3,\nu,i}(x_2)I\big)^lP_{i,0}(x_2)}{\big(w+x_{3,0,j}(x_2)-x_{3,\nu,i}(x_2)\big)^{l+1} }. \nonumber	
%
\end{eqnarray}\normalsize

We substitute \eqref{inverse for 0} into \eqref{variation2} and evaluate the residues of the coefficients for $\Delta$ and $\Delta^2$. The details of this evaluation are similar to the ones in the case $\lambda \neq 0$ and are omitted.

The proof of Theorem \ref{first two moments} is complete.

\vspace{.2cm}

\vspace{.5cm}

Our next result shows that under the assumptions of Theorem \ref{first two moments} the limit component projections coming out of the joint spectrum 
$\sigma_p(A_1,A_1A_2)$ coincide with $\Po_{j,\lambda}$.  

In a similar way as in our previous consideration, it follows from the implicit function theorem that if $\sigma(A_1,A_1A_2)$ satisfies conditions \ a) and b), then for each spectral component of $\sigma_p(A_1,A_1A_2)$ passing through $(1/\lambda,0)$ the first coordinate is expressed as an analytic function of the second one. To distinguish from the previous case when we considered the pair $(A_1,A_2)$, we denote the coordinates of a spectral point in $\sigma_p(A_1,A_1A_2)$ by $(z_1,z_2)$, so that the in the $j$-th component passing through $(1/\lambda,0)$, \  $z_1=z_{1,\lambda,j}(z_2)$. Similarly,
 we denote by $Q_{j,\lambda}(z_2)$ the projection
 $$Q_{j,\lambda}(z_2)=\frac{1}{2\pi i} \int_{|w-1|=\tilde{\delta}_j} \bigg(w-z_{1,\lambda,j}(z_2)A_1-z_2A_1A_2\bigg)^{-1}dw, $$
where the contours $\{w: \ |w-1|=\tilde{\delta}_j\}$  separates 1 from the rest of the spectrum of $\big(z_{1,\lambda,j}(z_2)A_1+z_2A_1A_2\big)$, as it was in \eqref{projection}. Since $A_1$ is normal, Theorem \ref{limit_projections} implies that there are limits ${\mathcal Q}_{j,\lambda}$ of $Q_{j,\lambda}(x_2)$ as $x_2\to 0$. 

\begin{lemma}\label{same_projections}
Suppose that pairs of matrices $(A_1,A_2)$ and $(A_1,A_1A_2)$ satisfy the conditions of Theorem \ref{first two moments}, and let $\lambda\in \sigma(A_1), \ \lambda\neq 0$. Then ${\mathcal P}_{j,\lambda}={\mathcal Q}_{j,\lambda} $.
\end{lemma}

\begin{proof}
 
 First we observe that relation \eqref{moment1} (the first relation in Theorem \ref{first two moments}) implies that for $j\in I_\lambda$  the compression of $A_2$ to  the range of $\Po_{j,\lambda}$ is $-x_{1,\lambda,j}^\prime(0)I_{{\mathcal R}(\Po_{j,\lambda})}$, where  $I_{{\mathcal R}(\Po_{j,\lambda})}$ is the identity matrix on the range of $\Po_{j,\lambda}$. This implies that $(-x_{1,\lambda,j}^\prime(0)), \ j\in I_\lambda$ are the eigenvalues of the compression of $A_2$ to the $\lambda$-eigenspace of $A_1$, which by \eqref{sum_of_projections} is the sum of the ranges of $\Po_{j,\lambda}$. Each of these ranges is of dimension 1, and by condition b)  all numbers $(x_{1,\lambda,j}^\prime(0))$ are different. Let $e_j, \ j\in I_\lambda$ be an eigenvector for $\Po_{j,\lambda}A_2\Po_{j,\lambda}$. Then these eigenvectors form a basis of the $\lambda$-eigenspace of $A_1$. 
 
 Similarly, applying Theorem \ref{first two moments} to the pair $(A_1,A_1A_2)$ we obtain that the compression of $A_1A_2$ to the $\lambda$-eigenspace of $A_1$ has $(-z_{1,\lambda.j}^\prime(0))$ as eigenvalues of multiplicity one each. 
   
    Propositions \ref{independent_projections} and \ref{relations_between_projections} imply
$${\mathscr P}_\lambda A_1A_2{\mathscr P}_\lambda= \left (\sum_{j=1}^{k_\lambda} {\mathcal P}_{j,\lambda} \right )A_1A_2\left (\sum_{j=1}^{k_\lambda} {\mathcal P}_{j,\lambda} \right )=A_1\sum_{j=1}^{k_\lambda} {\mathcal P}_{j,\lambda}A_2{\mathcal P}_{j,\lambda},$$
so that for $j\in I_\lambda$
$${\mathscr P}_\lambda A_1A_2{\mathscr P}_\lambda e_j=A_1\Po_{j,\lambda}A_2\Po_{j,\lambda}e_j=-x_{1,\lambda,j}^\prime(0)A_1e_j=-\lambda x_{1,\lambda,j}^\prime(0)e_j. $$
This shows that $-\lambda x_{1,\lambda,j}^\prime(0), \ j\in I_\lambda$ form  the spectrum of the compression of $A_1A_2$ to the $\lambda$-eigenspace of $A_1$ and that $e_j, \ j\in I_\lambda$ form the corresponding eigenbasis. Of course, this implies the statement we are proving. 
\end{proof}
\vspace{.3cm}

\Large \textit{Proof of Theorem \ref{A^2}}\normalsize

\vspace{.3cm}
Let $\lambda\neq 0$. Apply Theorem \ref{first two moments} to each of the pairs $(A_1,A_2)$ and $(A_1,A_1A_2)$ and use  Lemma \ref{same_projections}. We obtain
\large \begin{eqnarray*}
&  \Po_{j,\lambda}A_2T_\lambda A_2	\Po_{j,\lambda}=-\frac{x_{1,\lambda,j}^{\prime \prime}(0)}{2}\Po_{j,\lambda}, \\
&   \Po_{j,\lambda}A_1A_2T_\lambda A_1A_2\Po_{j,\lambda}=-\frac{z_{1,\lambda,j}^{\prime \prime}(0)}{2}\Po_{j,\lambda}.
\end{eqnarray*}\normalsize
It follows from the definition of the operator $T_\lambda$, \eqref{T}, that
\large$$T_\lambda A_1=A_1T_\lambda = -\Big (\sum_{\mu \in \sigma(A_1), \mu\neq \lambda} {\mathscr P}_\mu \Big ) +\lambda T_\lambda=-I+{\mathscr P}_\lambda+\lambda T_\lambda.$$\normalsize	
Since $\Po_{j,\lambda}A_1=\lambda \Po_{j,\lambda}$, this implies
\large\begin{eqnarray*}
&-\lambda \Po_{j,\lambda}A_2^2\Po_{j,\lambda} +\lambda \Po_{j,\lambda}A_2{\mathscr P}_\lambda A_2\Po_{j,\lambda}+\lambda ^2 \Po_{j,\lambda}A_2TA_2\Po_{j,\lambda} \\
&=-\frac{z_{1,\lambda,j}^{\prime \prime}(0)}{2}\Po_{j,\lambda}, \\
& \\
&-\lambda \Po_{j,\lambda}A_2^2\Po_{j,\lambda} +\lambda \Po_{j,\lambda}A_2	\Big(\sum_{l\in I_\lambda} \Po_{l,\lambda}\Big)A_2\Po_{j,\lambda}=\frac{\lambda ^2x_{1,\lambda,j}^{\prime \prime}(0)-z_{1,\lambda,j}^{\prime \prime}(0)}{2}\Po_{j,\lambda},\\
& \\
&-\lambda \Po_{j,\lambda}A_2^2\Po_{j,\lambda}+\lambda \Po_{j,\lambda}A_2\Po_{j,\lambda}A_2\Po_{j,\lambda}=\frac{\lambda ^2x_{1,\lambda,j}^{\prime \prime}(0)-z_{1,\lambda,j}^{\prime \prime}(0)}{2}\Po_{j,\lambda}. \\
\end{eqnarray*}\normalsize

Now, $\Po_{j,\lambda}^2=\Po_{j,\lambda}$ results in
$$\Po_{j,\lambda}A_2\Po_{j,\lambda}A_2\Po_{j,\lambda}= \Big (\Po_{j,\lambda}A_2\Po_{j,\lambda} \Big )^2, $$
and \eqref{moment1} yields
$$ \Po_{j,\lambda}A_2^2\Po_{j,\lambda}=\frac{z_{1j}^{\prime \prime}(0)+2\lambda\big(x_{1j}^\prime(0)\big)^2-\lambda ^2x_{1j}^{\prime \prime}(0)}{2\lambda}\Po_{j,\lambda}.$$
The proof is complete.

\vspace{.5cm}

\section{\textbf{Application to representations of Coxeter groups.\\ Proof of Theorem \ref{for Coxeter}}}\label{Coxeter}

\vspace{.3cm}
Recall that a Coxeter group is a finitely generated group $G$ with generators 
$g_1,\dots,g_n$ satisfy the following relations:
\[
(g_ig_j)^{m_{ij}}=1, \quad i,j=1,\dots,n,
\]
where $m_{ii}=1$ and $m_{ij}\in {\mathbb N}\cup \{ \infty \}$, \ 
 $m_{ij}\geq 2$ when $i\neq j$. 
It is easy to see that to avoid redundancies we must have $m_{ij}=m_{ji}$, 
and that $m_{ij}=2$ means $g_i$ and $g_j$ commute. 
The set of generators $\{g_1,\dots, g_n\}$ is called a \emph{Coxeter set 
of generators}, and  $m_{ij}$ are called the \emph{Coxeter exponents}. The matrix $\big \Vert m_{ij}\big \Vert $ is called a \textit{Coxeter matrix}. 
A Coxeter group with 2 generators is called a \textit{Dihedral} group. The monographs \cite{BB}, \cite{GP}, and \cite{Hu} are good sources for 
 information on Coxeter groups.

\vspace{.2cm}

Let $G$ be a Coxeter group with Coxter generators $g_1,...,g_n$, and $\rho: G\to V$ be a finite dimensional unitary representation of $G$. Suppose that $A_1,....,A_n$ is a tuple of $N\times N$ matrices. In this section we investigate what information about relations between $A_1,\dots,A_n$ can be obtained from the fact that the joint spectrum $\sigma_p(\rho(g_1),\dots,\rho(g_n))$ is contained in $\sigma_p(A_1,...,A_n)$.

Of course, for every pair of generators $g_i,g_j$ the representation $\rho$ generates a unitary representation of the Dihedral group generated by $g_i$ and $g_j$. It is well-known that every irreducible unitary representation of a Dihedral group is either 1- or 2-dimensional, and by Maschke's Theorem (\cite{M1}, \cite{M2}) that every representation is a sum of irreducible ones. The one dimensional representations of a Dihedral group are:
$$\rho(g_1)=\rho(g_2)=I , \mbox{or} \ \rho(g_1 )=\rho(g_2)=-I$$
 for an odd order group. For an even order group there is an additioal one-dimensional representation
 $$\rho(g_1)=I, \ \rho(g_2)=-I. $$
 Two-dimensional irreducible representations are equivalent to those generated by
 $$ \rho(g_1)=\left [ \begin{array}{rr} 1&0 \\ 0&-1 \end{array} \right ], \ \rho(g_2)=\left [ \begin{array}{rr} \cos \alpha & \sin \alpha \\ \sin \alpha & - \cos \alpha \end{array}\right ],$$
 where $0<\alpha<\pi$. If the group is finite, $\alpha $ is a rational multiple of $\pi$.
 It is easy to see (cf \cite{CST}) that the proper joint spectrum of images of the Coxeter generators of a Dihedral group under an irreducible representation could be either a line of the form (one-dimensional)
$$ \{(x_1,x_2): \ x_1\pm x_2=\pm 1\},$$
or a ``complex ellips" (two-dimensional)
$$\{(x_1,x_2): \ x_1^2+2(\cos \alpha) x_1x_2 +x_2^2=1\},$$ 
and the joint spectrum of $(\rho(g_1)$ and $\rho(g_1)\rho(g_2))$ could be
$$\{x_1\pm x_2=\pm 1\}  $$
for one-dimensional representations, and 
$$\{x_1^2-x_2^2+2(\cos \alpha)x_2=1\} $$ 
for two-dimensional.

\vspace{.2cm}

\vspace{.3cm}

\Large \textit{Proof of Theorem \ref{for Coxeter}}\normalsize

\vspace{.3cm}

 Consider $2\leq i\leq n$. Condition $(\ast)$ in the statement of Theorem \ref{for Coxeter} implies that every line and ellipse in the joint spectrum of $\rho(g_1)$ and $\rho(g_i)$ has multiplicity one, and, therefore, by condition (II) in  Theorem \ref{for Coxeter},  \ 
$\sigma_p(A_1,A_i)$ and $\sigma_p(A_1,A_1A_i)$   satisfy Theorems \ref{first two moments} and \ref{A^2} at $(\pm 1,0)$. Details of the local analysis near each of them are similar, so we concentrate on $(1,0)$. Following the notations of the previous section we denote by $\Po_{j,1}$ the component projections from Theorem \ref{limit_projections}. We will apply Theorem \ref{A^2}.

First suppose that the $j$-th component of $\sigma_p(A_1,A_i)$ which passes through $(1,0)$ is the line
$$x_1\pm x_2=1. $$
Then the corresponding component of the joint spectrum of $A_1$ and $A_1A_2$ is also a line, and  $x_{1j}(x_2)=1\pm x_2$, \ $z_{1j}(x_2)=1\pm z_2$, so that $x_{1j}^{\prime \prime}=z_{1j}^{\prime \prime}\equiv 0, \ (x_{1j}^\prime)^2\equiv 1 $. Theorem \ref{A^2} now implies
\begin{equation}\label{1}
\Po_{j,1}A_2^2\Po_{j,1}=\Po_{j_1}.
\end{equation} 
Let the $j$-th component is the ellipse
$$x_1^2+2\Big( \cos \alpha \Big) x_1x_2+x_2^2=1. $$
Then the corresponding component of $\sigma_p(A_1,A_1A_2)$ is given by
$$z_1^2-z_2^2+2\Big (\cos \alpha \Big)z_2=1. $$
In this case 
\begin{eqnarray*}
&x_{1j}^\prime(0)=z_{1j}^\prime (0)=-\cos \alpha, \ x_{1j}^{\prime \prime}(0)=-1+\cos ^2 \alpha, \\ 
&z_{1j}^{\prime \prime}(0)=1-\cos ^2 \alpha, \end{eqnarray*}
and Theorem \ref{A^2} for $\lambda=1$ shows that \eqref{1} holds in this case too. 

Thus, \eqref{sum_of_projections} shows that the compression of $A_2^2$ to the 1-eigenspace of $A_1$ is the identity. Since $A_1$ is normal, the projection ${\mathscr P}_1$, onto this subspace is orthogonal. Since the norm of $A_2$ is equal to 1, this implies that every 1-eigenvector of $A_1$ is a 1-eigenvector of $A_2^2$. 

A similar proof shows that every (-1)-eigenvector of $A_1$ is a \newline 1-eigenvector of $A_2^2$. 

Further, every component of the joint spectrum of Coxeter generators under a representation of a Dihedral group, either an ellipse, or a straight line, passes the same number of times through $(\pm 1,0)$ and through $(0,\pm 1)$. That is, if $t_1$ and $t_2$ are the multiplicities of 1 and -1 as eigenvalues of $\rho(g_1)$, and $u_1$ and $u_2$ are the same numbers for $\rho(g_2)$, then 
$$t_1+t_2=u_1+u_2. $$
Now, condition II) in the statement of Theorem \ref{for Coxeter} implies that the sum of multiplicities of eigenvalues 1 and -1  of $A_1$ and $A_i$ are the same. This sum is equal of the sum of dimensions of all Jordan cells in the Jordan representation of $A_i$ corresponding to eigenvalues $\pm 1$ and, of course, to the multiplicity of eigenvalue 1 for $A_i^2$. Thus, the sum of 1- and (-1)-eigenspaces of $A_1$ is exactly the invariant subspace for $A_i^2$ corresponding to eigenvalue 1, and the restriction of $A_i^2$ to this subspace is the identity. Since the square of a Jordan cell which corresponds to a non-zero eigenvalue and has  dimension higher than 1 is never diagonal, we see that  $\pm 1$-eigenvectors for $A_1$ and $\pm 1$-eigenvectors of $A_i$ span the same subspace, which we call $L$, that is invariant under the action of both $A_1$ and $A_i$. As mentioned above, every component of the joint spectrum of $\rho(g_1)$ and $\rho)g_2$ passes through $(\pm 1,0)$, so that the dimension of $L$ is equal the dimension of $\rho$. Of course, $L$ being spanned by $\pm 1$-eigenvectors of $A_1$ is independent of $i=2,...,n$, and 1) is proved.

The fact that $A_i\Big|_L$ are unitary and self-adjoint is straightforward. Indeed, it follows from  $L$ being spanned by $\pm 1$-eigenvectors of $A_i$ and from a simple fact that for an operator of norm 1 any two eigenvectors corresponding to different unimodular eigenvalues are orthogonal.
Since every component of the joint spectrum of Coxeter generators of a representation of a Coxeter group passes through $(\pm 1,0,...,0)$, we see that \eqref{same spectra} holds. The fact that restrictions of $A_i\Big|_L$ generate a representation of $G$ follows from the Theorem 1.1 in \cite{CST} stating that for a representation of a pair of unitary matrices $U_1,U_2$ their joint spectrum determines a number m such that $(U_!U_2)^m=1$ ($m$ is infinite, if there is no such relation). Thus, the Coxeter matrices of $\rho$ and $\tilde{\rho}$ are the same, and 2) is proved.

Finally, 3) follows from \eqref{same spectra} and Theorem 1.2 in \cite{CST}.

We are done.


\end{document}